\documentclass[article, leqno]{memoir}
\usepackage{amsmath, amsthm, amsfonts, amssymb}

\usepackage{tikz} \usetikzlibrary{cd} 

\midsloppy 

\settypeblocksize{47\baselineskip}{360pt}{*} 
\setulmargins{1.5in}{*}{*} 
\setlrmargins{*}{*}{1} 
\setheadfoot{\onelineskip}{2\onelineskip} 
\setheaderspaces{*}{1.5\onelineskip}{*}  

\checkandfixthelayout

\theoremstyle{plain}

\newtheorem{proposition}{Proposition}

\newtheorem{corollary}[proposition]{Corollary}
\newtheorem*{thm}{Theorem}

\newtheorem{lemma}[proposition]{Lemma}

\newtheorem{satz}{Satz}

\theoremstyle{definition}

\theoremstyle{remark}

\newtheorem*{remark}{Remark}

\newcommand{\defeq}{\stackrel{\mathrm{def}}{\, = \,}}

\newcommand{\bZ}{{\mathbb{Z}}}
\newcommand{\bR}{{\mathbb{R}}}
\newcommand{\bC}{{\mathbb{C}}}
\newcommand{\bQ}{{\mathbb{Q}}}
\newcommand{\bF}{{\mathbb{F}}}

\newcommand{\fp}{{\mathfrak{p}}} 
\newcommand{\fP}{{\mathfrak{P}}} 
\newcommand{\cO}{{\mathcal{O}}} 
\newcommand{\cC}{{\mathcal{C}}} 
\newcommand{\mf}{\mathfrak} 
\newcommand{\mc}{\mathcal} 

\begin{document} 

\title{Artin's First Article on the Artin $L$-Series (1924): \emph{Paraphrasis} and Commentary}

\author{Commentary by W. E. Aitken}

\date{December 2022 Edition\thanks{Copyright \copyright \ 2022 by Wayne Edward Aitken.
Version of \today. 
This work is made available under a Creative Commons Attribution 4.0 License. 
Readers may copy and redistribute this work under the terms of this license. Thanks to Jason Martin for his comments and suggestions on earlier drafts.}}

\maketitle

This document started as a kind of mathematically oriented freestyle translation of  E.~Artin's \emph{\"Uber eine neue Art von $L$-Reihen} (``A New Kind of $L$-Series'')~\cite{Artin1924}
with added commentary.\footnote{Artin's 
article is 20 pages long. As the reader will notice, this version has gained some length since it departs from Artin's elegant but succinct style, and includes  a
generous amount of commentary.} 
However, I made enough changes  in notation, terminology and even in
the details of  proofs that the term \emph{translation} is perhaps not entirely appropriate. Instead, I
decided to borrow the
Latin term  \emph{paraphrasis} which was in turn borrowed from Greek. This term connotes a very close connection to the original article, while giving me
license to make adaptations here and there for the benefit of a modern reader. It also allows me to focus on  the  mathematical description and development without
the responsibility capturing the subtle nuances of German mathematical writing style from almost one hundred years ago.\footnote{A good thing too; my knowledge of German and the mathematical conventions of the period is not good enough for a faithful literal translation. Readers with a strong historical interest are encouraged to read this paraphrasis alongside the German original to get a fuller picture.}
This license allows me to sneak in clarifying devices, such as commutative diagrams, which were not current in the 1920s.

 My goal is to give a modern reader access to the mathematics of~\cite{Artin1924}, but not necessarily to follow the stylistic and mathematical conventions 
 of the 1920's. For example, in his article Artin does not use the usual notation $\mc O_k$ for the ring of integers of a number field $k$, but 
 instead just speaks of integers in $k$. He speaks of prime ideals in $k$ instead of in the the ring of integers~$\mc O_k$. This is just one of several conventions that
 I have chosen not to follow in this \emph{paraphrasis}. In one case I have even changed the statement of a result, and the proof, to be a bit more
 general since it was easy to do so using Artin's methods.
 I have even added section titles to supplement Artin's  simple section numbering, and I have added a bibliography.
  So the reader should expect these sorts of changes from Artin's original. However, I hope these changes are in a mathematical sense all minor, and that I have captured the spirit of things well. My intent is to  open up this great landmark of mathematics to modern number theorists and transport the reader to the genesis of the all important Artin L-Series. I have been faithful the order of presentation; the sectioning and section numbers, numbered equations, and the numbering of main results (\emph{S\"atze})  are all faithful to~\cite{Artin1924}.

In this paper Artin introduces what are known as \emph{Artin $L$-series}, but the paper has much more. It has the first statement, and proofs in many  cases, of
\emph{Artin's Reciprocity Law}, arguably the most important result in class field theory. It has the analytic continuation and functional equations for these new L-series: actually he gives meromorphic continuations but only for powers of the $L$-series. So the analytic continuations he gives here are viewed as possibly ``multivalued''. He conjectures that these are single-valued, i.e., that the $L$-series are indeed meromorphic on $\bC$, and proves this in the significant case where the Galois group is $A_5$ (the Icosahedral case). Note that the Abelian case of the single-valuedness claim follows from Hecke's earlier work together with Artin's reciprocity law; the general case was handled much later by Brauer in 1947. Artin goes further and conjectures that primitive $L$-series are
entire (holomorphic) when they are not equal to Dedekind zeta functions, and gives some evidence in the the $A_5$ case. This conjecture remains open, even for the specific case of~$A_5$ extensions.
Finally, Artin gives a proof (assuming the Reciprocity Law) of the Chebotaryov density theory. Unknown to Artin, at about the same time as Artin was writing~\cite{Artin1924}, Nikolai Chebotaryov proved this same result, a conjecture of Frobenius, but with a different method that, in an interesting twist, would be the inspiration for Artin's definitive proof of the Reciprocity law of 1927. So all in all, this is an amazingly rich and interesting paper.

The paper, published in 1924, reflects a seminar in Hamburg in  July 1923. 
This was the initial presentation of the theory, but is not the final word on the birth of the theory.
It suffered from gaps that were soon fixed by Artin himself during his Hamburg years:
\begin{enumerate}
\item
It depended on a general reciprocity law that Artin did not prove until 1927. This is the partially proved Satz 2 in the current paper.
\item
The ramified primes were not suitably handled yet. This was fixed in 1930, with a factor for the ``infinite prime'' and a theory of conductors.
(See \cite{Artin1931})
\end{enumerate}
The work on conductors made an impact on the number theorist Hasse who was essentially 
the same age as Artin, and inspired Noether's work on her theorem on normal integral bases. 
So the work on reciprocity and conductors makes this work of interest beyond concern for Artin $L$-series per se.

At the time of this paper, Artin was just starting his mathematical career. He received his PhD under Herglotz in 1921, in Leipzig.
After a year at G\"ottingen he accepted  a permanent position in Hamburg in 1922  and stayed there for 15 years. It was a very rich and productive time in Artin's career
and which came to an end when Artin moved to the United States to escape the Third Reich.
(See~\cite{Roquette2003} for additional historical perspective.)


\chapter*{``Concerning a new kind of L-Series'': Introduction}

By E. Artin in Hamburg.

\medskip

{\color{blue} 
In what follows the black paragraphs give my very free translation of Artin's original paper. The blue paragraph gives my notes and comments.
A similar convention will hold for footnotes.

We start with the with the introduction, which consists of a short paragraph:}

\medskip

For investigating non-Abelian algebraic number fields one needs a new kind of $L$-series
that generalizes the usual $L$-series for Abelian algebraic number fields. 
These analytic functions are formed with Frobenius style group characters.
This article is dedicated to the investigation of such functions.

\chapter{Frobenius Style Group Characters: Review}

For the convenience of the reader, I will begin by briefly giving the  formulas and notation that we will need from
the theory of group characters.\footnote{See J. Schur 1905, \emph{Neue Begr\"undung der Theorie der Gruppencharaktere} (New foundation for the theory of group characters),
Sitzungsberichte (conference reports), Berlin, and Speiser~\cite{SpeiserGroups} Chapters~10-12.}

Let $G$ be a finite group of order~$n$. Decompose $G$ into $x$ conjugacy classes~$C_1, \ldots, C_x$,  and let
$h_i$ be the number of elements of $C_i$.

Let $\Gamma$ be a representation of the group $G$ as nonsingular matrices. Given $\Gamma$ we get a \emph{character}~$\chi$
which is a function $G\to \bC$ that assigns to $\sigma \in G$ the trace of the associated matrix.
There are~$x$ irreducible representations $\Gamma_1, \ldots, \Gamma_x$, and let $\chi^1, \ldots, \chi^x$ be their associated characters.
These characters are called \emph{simple characters}. Every character $\chi$ is in fact the linear combination of simple characters:
\begin{equation}\label{E1}
\chi(\sigma) = \sum_{i=1}^x r_i \chi^i(\sigma)
\end{equation}
where $r_i$ are nonnegative integers associated with the decomposition of  $\Gamma$ into irreducible representations.

The simple characters satisfy the following formulas
\begin{equation}\label{E2}
\sum_{\sigma} \chi^i (\sigma) \chi^k(\sigma^{-1}) = n \delta_{i k}
\end{equation}
and
\begin{equation} \label{E3}
\sum_{i=1}^x \chi^i(\sigma) \chi^i (\tau^{-1}) 
= 
\begin{cases} 0 \text{ if $\sigma$ and $\tau$ are in different classes}, \\ \frac{n}{h_r} \text{ if $s$ and $\tau$ are both in the class $C_r$}. \end{cases}
\end{equation}

Furthermore, suppose $H$ is a subgroup of $G$ and that
\begin{equation}\label{E4}
G = \sum_{i=1}^s H S_i
\end{equation}
is the decomposition into cosets (here $S_i \in G$).

Let $\Delta$ be a representation of the subgroup $H$ of degree $\delta$, and let $A_\sigma$ be the matrix associated to $\sigma \in H$.
If $\sigma \in G$ is not in $H$ we take $A_\sigma$ to be the zero matrix.  We build the matrix $B_\sigma$ out of blocks in the following way:
\begin{equation}\label{E5}
B_\sigma = \left(  
A_{S_i \sigma S_k^{-1} }\right).
\end{equation}
As stated this is  an $s$ by $s$ square matrix with entries equal to $\delta$ by $\delta$ square matrices, where $s$ is the index of $H$ in~$G$.
But we regard $B_\sigma$ as defining a square~$s\delta$ by $s \delta$ matrix, and it turns out that this gives a representation of $G$
called the representative of $G$ induced by the representation of $H$.\footnote{See Speiser~\cite{SpeiserGroups} \S 52
from  which formula (44) can easily be easily derived.
See also an 1898 report by Frobenius called \emph{\"Uber Relationen zwischen den 
Charakteren einer Gruppe und denen ihrer Untergruppen} (Concerning the  connection between the characters of a group
and those of its subgroups).}

{\color{blue}
\begin{remark}
In the above $A_{S_i \sigma S_k^{-1}}$ designates the $(i, k)$ block  (using the $i$th row partition and $k$th column partition). There are $s^2$ such blocks total, and each block is a~$\delta$ by $\delta$ matrix. 
So the induced representation is given concretely in terms of~$s \delta$ by $s \delta$ matrices associated to each $\sigma \in G$.

Artin is viewing the group acting on the right of vectors. If we act on the left, which is common today,
we end up with the $(i, k)$ block looking like $A_{S_i^{-1} \sigma S_k}$.
\end{remark}
}

If $\psi$ is the character of 
the representation $\Delta$ then the character $\chi_\psi$ associated to the representation~(\ref{E5}) is called the \emph{character of $G$ induced by the 
character $\psi$
of~$H$}.

Let $\psi_1, \ldots, \psi_\lambda$ be the simple characters of the subgroup~$H$.
Then we can express the restriction of $\chi^i$ to $H$ as a nonnegative integral linear combination of $\psi_1, \ldots, \psi_\lambda$
with nonnegative integer coefficients $r_{1 i}, \ldots, r_{\lambda i}$:
\begin{equation} \label{E6}
\chi^i(\tau) = \sum_{\nu = 1}^\lambda r_{ \nu i} \psi_\nu (\tau)\qquad (i=1, \dots, x)
\end{equation}
for all $\tau \in H$.
Similarly, we can express the induced character $\chi_{\psi_i}$
as a nonnegative integral linear combination of the simple characters $\chi_1, \ldots, \chi_x$ of $G$,
and in fact the nonnegative coefficients are just the coefficients that arise in (\ref{E6}):
\begin{equation}  \label{E7}
\chi_{\psi_i} (\tau) = \sum_{\nu =1}^x r_{i \nu} \chi^{\nu} (\tau) \qquad (i=1, \dots, \lambda).
\end{equation}
for all $\tau \in G$.


{\color {blue} 
\begin{remark}
The above is an expression of the Frobenius reciprocity law.
 The version of Serre~\cite{Serre1977} Section 7.1  can be written
 $$
 \left< \psi, \mathrm{Res}\, \chi  \right>_H
 =
  \left<\mathrm{Ind}\,\psi,  \chi  \right>_G.
 $$
The above statement can be derived from this.
 \end{remark}
}

\chapter{Construction of the $L$-Series}

From now on let $k$ be an algebraic number field, let $K$ be a Galois extension of $k$, and let~$G$ be the Galois group of $K / k$.

Let $\mf p$ be a prime ideal in the ring of integers of $k$ not dividing the relative discriminant of $K/k$. Let $\mf P$ be a prime
ideal of $\mc O_K$ dividing $\mf p \mc O_K$.

We chose an element $\sigma \in G$ such that for all algebraic integers $A$ in $K$ we have
\begin{equation}\label{E8}
\sigma A \equiv A^{N \mf p} \pmod {\frak P}
\end{equation}
where $N\frak p$ is the norm of $\frak p$ in $k$.
For the existence of such a $\sigma$ see Weber's Algebra~\cite{Weber2v1899}, \S 178 (volume 2).

This congruence determines $\sigma$ uniquely given a choice of $\frak P$,  since if $\sigma_1$ satisfies the same congruence then, for all algebraic integers $A$ in $K$,
$$
\sigma^{-1} \sigma_1 A \equiv A  \pmod {\frak P},
$$
and so $\sigma^{-1} \sigma_1$ belongs to the inertia group (Tr\"agheitsgruppe) of $\frak P$. By our assumption ($\frak p$ not dividing the relative discriminant)
the inertia group is trivial.

Next suppose one chooses $\frak P'$ instead of $\frak P$
as a designated prime divisor of $\mf p \mc O_K$. Since~$G$ acts transitively on primes above $\frak p$, we have
$
\tau \frak P = \frak P'
$
for some $\tau \in G$. It is easy to check that one gets~$\tau \sigma \tau^{-1}$  as the corresponding element of $G$ (where $\sigma$ is the 
corresponding element for $\frak P$).

So we have a way to associate to $\frak p$ a well-defined conjugacy class $C$ of $G$. 
It is well-known that each element of $C$ generates the decomposition group
for some~$\frak P$ above $\frak p$ but this property does not in general completely determine the class $C$ (in fact certain powers
of this class $C$ with have this property).\footnote{\label{FrobFoot}This assignment of conjugacy classes to  prime ideals was
already carried out by Frobenius. See the 1896 Berlin report called \emph{\"Uber Beziehungen zwischen Primidealen eines algebraischen K\"orpers und den Substitutionen 
seiner Gruppe} (concerning the relationships between prime ideals of an algebraic field and the elements of its Galois group).}
We will say that the prime ideal $\frak p$ \emph{belongs to the class $C$} and we will write this class as $C_{\frak p}$.

{\color {blue} 
\begin{remark}
We call each element of $C_{\frak p}$ a \emph{Frobenius element}, and the class as a whole the \emph{Frobenius class},
in honor of Frobenius who, as Artin points out in the footnote, developed this idea earlier.
Artin does not really use these terms in the German original of this paper, but I will use them in the translation below for the convenience of the modern reader.
 \end{remark}
}

From now on let $\Gamma$ be a linear representation of $G$. For $\frak p$ as above let $A_{\frak p}$ be a matrix associated to an element of $C_{\frak p}$ via $\Gamma$.
Since the elements of $C_{\frak p}$ are conjugate, the characteristic polynomial 
$$
\left| E - t A_{\frak p} \right|
$$ of $A_{\frak p}$ does not depend on the choice of $A_{\frak p}$. Here $E$ is the identity matrix and, as usual, the absolute values indicates determinant.
Note that $A_p$ will change by a conjugate if~$\Gamma$ is replaced by an equivalent representation, so the characteristic polynomial only depends
on the representation $\Gamma$ up to equivalence.

We define the associated $L$-series by the formula 
\begin{equation}\label{E9}
L(s, \chi; k) = \prod_{\mf p} \frac{1}{| E - \left(N \mf p\right)^{-s} A_{\mf p} |}
\end{equation}
where  $s$ is a complex variable and $\chi$ denotes the character associated with the representation $\Gamma$.
Here, the product varies only for the set of prime ideals $\mf p$ of $\mc O_k$ that do no divide the relative discriminant of $K/k$.

{\color {blue} 
\begin{remark}
In a later paper, Artin gives an explicit formula for terms associated to primes that do divide the relative discriminant of $K/k$. 
Note that the above $L$ series is expressed using $\chi$ instead of $\Gamma$, since $\chi$ determines $\Gamma$ up to equivalence and
so determines the expression on the right-hand side of (\ref{E9}).
 \end{remark}
}

The function $L(s, \chi; k)$ converges absolutely and uniformly on any closed and bounded region in the half plane $\frak R (s) > 1$. To see this observe
that every root of the characteristic polynomial $\left| E - t A_{\frak p} \right|$ is a root of unity. Thus $L(s, \chi; k)$ is a product of terms of the form
$$
\frac{1}{1 - \left(N \frak p\right)^{-s} \varepsilon }
$$
where $\varepsilon$ is a root of unity.

{\color {blue} 
\begin{remark}
Since $A_\mf p$ has finite order it is diagonalizable with eigenvalues all equal to roots of unity.
So its characteristic polynomial factors as described by Artin. 

Some of the convergence issues can be handled with the following well-known criterion: if an infinite series $\sum |a_i|$ converges then the corresponding infinite product~$\prod (1+a_i)$ converges, and the terms $1+a_i$ can be reordered freely with convergence to the same result. Furthermore, if each term $1+ a_i$ is nonzero then
the limit is nonzero. (See, for example,  \cite{Stein2}, Chapter 5, Proposition 3.1 for some justification.)

On the other hand, 
it might be convenient to wait on convergence issues until we have the formula for the logarithm given by Artin below. 
 \end{remark}
}

One can now expand (\ref{E9}) in a Dirichlet series and express the coefficients in terms of the character $\chi$.
The resulting formulas are not very clear (``Die Formeln werden aber wenig übersichtlich''). On the other hand, we arrive at a simple formula for the 
logarithm of~(\ref{E9}).

First we associate a conjugacy class $C_{\mf p^\nu}$ to any power $\mf p^\nu$ of a prime ideal $\mf p$. We simply take the class consisting of $A_{\frak p}^\nu$
where $A_{\frak p} \in C_{\frak p}$.
It is easy to see that this forms a conjugacy class of~$G$. We write
\begin{equation}\label{E10}
\chi \left( \frak p^\nu \right) = \chi (\sigma)
\end{equation}
where $\sigma$ is any member of $C_{\frak p^\nu}$.

Now let $\varepsilon_1, \varepsilon_2, \ldots, \varepsilon_f$ be the roots of the equation $\left| Et  -  A_{\frak p} \right| = 0$. Then
\begin{equation}\label{E11}
\chi \left( \frak p^\nu \right) = \varepsilon_1^\nu + \varepsilon_2^\nu + \ldots + \varepsilon_f^\nu.
\end{equation}
So we get for $|t| < 1$
$$
-\log|E - t A_{\frak p}| = - \sum_{i=1}^f \log (1 - t \varepsilon_i) = \sum_{i=1}^f \sum_{\nu=1}^\infty \frac{\varepsilon_i^\nu}{\nu} t^\nu
=
\sum_{\nu=1}^\infty \frac{\chi(\frak p^\nu)}{\nu} t^\nu
$$

{\color {blue} 
\begin{remark}
Here we understand $\log$ as a multivalued functions, or equivalently we regard some of our equations as being valid modulo $(2\pi  i ) \bZ$.
So for example, the first equation above can be regarded as a congruence modulo~$(2\pi  i )\bZ$.
 \end{remark}
}

This leads to the desired formula:
\begin{equation}\label{E12}
+ \log L(s, \chi; k) = \sum_{\frak p^\nu} \frac{\chi(\frak p^\nu)}{\nu \left(N \frak p^\nu \right)^{s}},
\end{equation}
where the sum varies over all powers of prime ideals of $k$ not dividing the relative discriminant of $K/k$.

{\color {blue} 
\begin{remark}
Associated convergence issues can be justified by the observation that 
$$
\sum_{\mf p^\nu} 
\left| 
\frac{\chi(\frak p^\nu)}
{\nu \left( N  \mf p^\nu \right)^s } 
\right| 
\; < \;
\sum_{M=1}^\infty
m \frac{f}
{  M ^{\sigma_0}} 
= mf \zeta(\sigma_0) < \infty
$$
assuming that $\Re(s) \ge s_0 > 1$. Here $\zeta(s)$ is the  classical Zeta function,
the left sum is taken over all ideals of the form $\mf p^\nu$, in any order, where $\mf p$ is a prime ideal of~$\mc O_k$ relatively prime to the relative
discriminant of $K/k$, and $\nu$ is a positive integer. Also
 $f$ is the degree of the representation $\Gamma$ and~$m$ is the degree~$[k\colon \bQ]$, so that at most $m$ ideals of the form $\mf p^\nu$ can share the same norm.
 
In particular we have the absolute convergence of
 $$
 \sum_{\mf p^\nu} 
\frac{\chi(\frak p^\nu)}
{\nu \left( N  \mf p^\nu \right)^s } 
$$
which justifies the manipulations above. We also get uniform convergence on the set $\Re(s) \ge s_0$ for each $s_0 > 1$, and so the sum gives
a homomorphic function on the set defined by~$\Re(s) > 1$.
By exponentiation we get the desired convergence properties for our Euler product expansion of $L$ as well, 
including the  invariance under reordering of terms with a product that defines a holomophic function in $s$
with no zeros on the set defined by $\Re(s) > 1$.
 \end{remark}
}

Either from (\ref{E9})
or even better from (\ref{E12}) one sees that 
\begin{equation}\label{E13}
L(s, \chi + \chi') = L(s, \chi) L(s, \chi')
\end{equation}
for any two characters $\chi$ and $\chi'$.

If $\chi$ is a simple character then we will call the associated $L$-series a \emph{primitive $L$-series}.
If $\chi$ is a general character expressed in terms of simple characters, as in~(\ref{E1}) then (\ref{E13}) gives us
\begin{equation}\label{E14}
L(s, \chi) = \prod_{i=1}^x \left( L(s, \chi^i) \right)^{r_i}.
\end{equation}

A brief remark about the dependence on the field $K$: suppose $\Omega$ is an extension of $K$ that is Galois over $k$.
Then $G = \mathrm{Gal}(K/k)$ is isomorphic to 
the quotient group~$\mathrm{Gal}(\Omega/k) / \mathrm{Gal}(\Omega/K)$. If $\sigma \in \mathrm{Gal}(\Omega/k)$ is
such that (\ref{E8}) is valid for all algebraic integers in $\Omega$ then it will of course be valid for all algebraic integers in~$K$.
Furthermore,  (\ref{E8}) will be valid for algebraic integers $A$ in $K$ if we replace $\sigma$ with any element of the coset~$\sigma \, \mathrm{Gal}(\Omega/K)$.
Next observe that every character of~$\mathrm{Gal}(\Omega/k) / \mathrm{Gal}(\Omega/K)$
is a character of $\mathrm{Gal}(\Omega/k)$, and every simple character of~$\mathrm{Gal}(\Omega/k) / \mathrm{Gal}(\Omega/K)$
is a simple character of $\mathrm{Gal}(\Omega/k)$. In particular every $L$-series using $K$ as the extension will essentially be an $L$-series using $\Omega$
as the extension, and if the $L$ series is primitive using $K$ then it will be primitive using $\Omega$.
However, the relative discriminant of~$\Omega/k$ may exclude a finite number of prime factors in the $L$-series that occur using the relative discriminant of~$K/k$.
But we will consider $L$-series that differ from each other by only a finite number of factors as being essentially the same.
By the way, we will be able to normalize the $L$-series later to be truly invariant of $K$.

\color{blue}
\begin{remark}
The above uses a fundamental compatibility principle for  of the Frobenius element associated with two extensions $\Omega/k$ and $K/k$ of a common base field $k$.  
This principle  is needed in several places in this paper, so I will go ahead and codify it as a lemma. 
 I will switch the roles of $K$ and $\Omega$ here since in what follows $\Omega$ is often used to denote an intermediate field.
\end{remark}

\begin{lemma} \label{triple_lemma}
Suppose $K/k$ is a Galois extension of number fields with Galois group~$G$ and let $\Omega$ be an intermediate field such that $\Omega/k$ is also Galois.
Let $\mf p$ be a prime ideal of $\mc O_k$ not dividing the relative discriminant of $K/k$, let $\mf q$ be a prime ideal of~$\mc O_\Omega$ above $\mf p$, and
let $\mf P$ be a prime ideal of $\mc O_K$ above $\mf q$. In other words we have a triple extension $K/\Omega/k$ with corresponding prime ideals $\mf P, \mf q, \mf p$.

Then if $\sigma \in G$ is the Frobenius element associated to $\mf P$, then the restriction~$\sigma'$ of $\sigma$ to $\Omega$ is the Frobenius element
of $\mf q$ in the Galois group of~$\Omega/k$. When we identify the Galois group of $\Omega/k$ with $G/H$
where  $H$ is the Galois group of $K/\Omega$, then this Frobenius element $\sigma'$
is the coset $\sigma H \in G/H$.
\end{lemma}

\begin{proof}
By (\ref{E8}) we have that $\sigma A - A^{N \mf p} \in \mf P$ for all $A \in \mc O_K$.
So
$$
\sigma' A - A^{N \mf p} \in \mf P \cap \mc O_\Omega = \mf q
$$
for all $A \in \mc O_\Omega$. Hence
\begin{equation*}
\sigma' A \equiv A^{N \mf p} \pmod {\frak q}
\end{equation*}
for all $A \in \mc O_\Omega$, and so $\sigma'$ is the desired Frobenius element.
By basic Galois theory,~$\sigma'$ corresponds to the coset $\sigma H \in G/H$.
\end{proof}
\color{black}

\chapter{The  Theorem on Induced Representations}

Let $H$ be a subgroup of $G$, let $\Omega$ be the subfield of $K$ fixed by $H$, so~$H$ is the Galois group of $K/\Omega$.

The first main theorem (Satz 1) concerns the following situation: 
\begin{itemize}
\item $\Delta$ is a representation of $H$.
\item $\Gamma_\Delta$ is the induced representation of $G$. 
\item $\psi$ is the character of $\Delta$, and $\chi_\psi$ is the character of $\Gamma_\Delta$.
\item Exclude as factors of $L(s, \psi; \Omega)$ any prime dividing the relative discriminant of $K/k$ (considered as an ideal of the 
ring of integers of $\Omega$).
\end{itemize}

\begin{satz} 
In the situation discussed above
\begin{equation}\label{E15}
L(s, \psi; \Omega) = L(s, \chi_\psi ; k).
\end{equation}
\end{satz}

\begin{proof}
We set up the following notation:
\begin{itemize}
\item
Let $\mf p$ be a prime ideal of $\cO_k$ not dividing the relative discriminant of $K/k$.
\item
Let  $\mf q_1,  \mf q_2, \ldots, \mf q_r$ be the prime ideals of $\cO_\Omega$ dividing $\fp \cO_\Omega$:
$$
\fp \cO_\Omega = \mathfrak q_1 \mathfrak q_2 \cdots \mathfrak q_r.
$$
\item
Let~$l_i$ be the relative degree of $\mathfrak q_i$ over $\fp$. In other words, $(N \fp)^{l_i}$ is the size~$N \mathfrak q_i$ of
the residue field~$\cO_\Omega/ \mathfrak q_i$.
\item
For each $\mf q_i$ choose a prime ideal $\fP_i$ of $\cO_K$ dividing $\mf q_i \cO_K$.
\item
For each such $\fP_i$ let $\tau_i \in G$ be chosen so that $\fP_i = \tau_i \fP_1$. (Recall that the Galois group $G$ acts transitively on
the primes of $\cO_K$ dividing $\mf p \cO_K$).
\item
Let $\sigma \in G$ be the Frobenius element associated with $\fP_1$ over $k$. In other words,
$$
\sigma A \equiv A^{N\frak p} \pmod {\fP_1}
$$
for all $A \in \cO_K$.
\item
Let $\sigma_i \defeq \tau_i \sigma \tau^{-1}_i$. Observe that
\begin{equation}\label{E16}
\sigma_i A \equiv A^{N\frak p} \pmod {\fP_i}
\end{equation}
for all $A \in \cO_K$
so $\sigma_i$ is the Frobenius element associated with~$\fP_i$. 
Thus $\sigma_i$ generates the decomposition group of $\fP_i$.
\end{itemize}

Claim: \emph{The Frobenius element associated with $\fP_i$ over $\Omega$ is equal to $\sigma_i^{l_i}$}.
To see this first observe that from (\ref{E16})
\begin{equation}\label{E17}
\sigma^{l_i}_i A \equiv A^{(N\frak p)^{l_i}}  \equiv A^{N \frak q_i} \pmod {\fP_i},
\end{equation}
So to establish that $\sigma_i^{l_i}$ is the Frobenius element we just need to show that $\sigma_i^{l_i} \in H$.
In the special case where $A = \alpha \in \cO_\Omega$ we have from (\ref{E17}) and Fermat's little theorem that
$$
\sigma^{l_i}_i \alpha \equiv \alpha^{N \frak q_i} \equiv \alpha \pmod {\fP_i}.
$$
 Since $\mf p$ does not divide the relative discriminant of~$K/k$, this means 
that~$\sigma^{l_i}_i \alpha = \alpha$ for all $\alpha \in \cO_\Omega$ and hence for all $\alpha \in \Omega$.
So $\sigma^{l_i}_i  \in H$ as desired.

\color{blue}
\begin{remark}
Note that $\sigma_i$ is in the decomposition group of $\mf P_i$, and so $\sigma_i^{l_i}$ is, of course, in this decomposition group.
Since $\mf P_i$ is unramifield over $\mf q_i$, the canonical map from the decomposition group of $\mf P_i$ to the Galois group of $\cO_K/\mf P_i$
is injective.
\end{remark}
\color{black}

Next we observe that $l_i$ is the smallest positive power $\nu$ of $\sigma_i$ such that $\sigma_i^\nu \in H$.
To see this observe that if $\sigma_i^\nu \in H$ then by (\ref{E16})
$$
\sigma^{\nu}_i \alpha = \alpha \equiv  \alpha^{(N\mathfrak p)^{\nu}} \pmod {\frak P_i}
$$
for all $\alpha \in \cO_\Omega$. Thus $(N\mathfrak p)^{\nu} \ge (N \fp)^{l_i} = N \mf q_i$ and so $\nu \ge l_i$.

\color{blue}
\begin{remark}
The last step becomes clear when we observe that every element of the residue field has been shown to be a root of $X^{(N\mathfrak p)^{\nu}} - X$ which 
is a polynomial in~$X$ of 
degree $(N\mathfrak p)^{\nu}$. But the residue field has  $(N \fp)^{l_i} = N \mf q_i$ elements.
\end{remark}
\color{black}

Claim: \emph{Consider cosets $H \sigma_{\nu}^a \tau_\nu$ and~$H \sigma_{\mu}^b \tau_\mu$. These cosets are equal if and only if
$\nu = \mu$ and~$a \equiv b \pmod {l_\nu}$.
}

One direction of this claim is straightforward since $\sigma_\nu^{l_\nu} \in H$, so if $a \equiv b$ modulo~$l_\nu$ then 
$H \sigma_{\nu}^a =H \sigma_{\nu}^b$ and so $H \sigma_{\nu}^a \tau_\nu = H \sigma_{\nu}^b \tau_\nu$.
For the other direction, assume that~$H \sigma_{\nu}^a \tau_\nu = H \sigma_{\mu}^b \tau_\mu$, and so $$\sigma_{\nu}^a\;  \tau_\nu = \tau_0 \; \sigma_{\mu}^b \; \tau_\mu$$
with $\tau_0 \in H$. So by the definition of $\sigma_\nu$ and $\sigma_\mu$
$$
\tau_0 = \sigma_{\nu}^a\;  \tau_\nu \; \tau_\mu^{-1}\; \sigma_{\mu}^{-b} = \tau_\nu \; \sigma^{a-b} \; \tau_\mu^{-1}
$$
and, since $\sigma$ is in the decomposition group of $\mf P_1$,
$$
\tau_0 \mf P_\mu =   \tau_\nu \; \sigma^{a-b} \; \tau_\mu^{-1} \mf P_\mu = \tau_\nu \; \sigma^{a-b} \mf P_1 = \tau_\nu \mf P_1  = \mf P_\nu.
$$
Since $\tau_0 \in H$, it is the identity map on $\mf P_\mu \cap \mc O_\Omega = \mf q_\mu$, but the image of $\mf P_\mu \cap \mc O_\Omega$
is $\mf P_\nu \cap \mc O_\Omega = \mf q_\nu$. So $\mf q_\mu = \mf q_\nu$. Thus $\mu = \nu$.  We 
then have $\sigma_{\nu}^a\;  \tau_\nu = \tau_0 \; \sigma_{\nu}^b \; \tau_\nu$ so that~$\sigma_\nu^{a-b} \in H$
which implies that $a \equiv b \pmod {l_\nu}$.

So we have identified $l_1 + \ldots + l_r$ distinct cosets of $H$. 
But we know that $$l_1 + \ldots + l_r = [\Omega \colon k] = [G: H].$$
So we have identified all the right cosets of $H$.

Note  that $H \sigma_{\nu}^a \tau_\nu = H \tau_\nu \sigma^a$, and so
have coset representations, as in (\ref{E4}) with $S_i$ varying in the sequence
$$
\tau_1, \tau_1 \sigma, \ldots, \tau_1 \sigma^{l_1-1}, \tau_2,  \tau_2 \sigma, \ldots \ldots, \tau_r, \tau_r\sigma, \ldots, \tau_r \sigma^{l_r-1}.
$$
In other words, each $S_i$ is of the form $\tau_\nu \sigma^a$ with $0\le \nu \le r$ and $0\le a < l_\nu$.

According to (\ref{E5}), in  the induced representation $\Gamma_\Delta$ of $G$, the element $\sigma \in G$ is represented by the matrix
described in terms of blocks as follows:
$$
B_\sigma = \left(  
A_{S_i \sigma S_k^{-1} }\right)
=
\left(  
A_{\tau_\nu \sigma^{a-b + 1} \tau_{\mu}^{-1} }
\right)
$$
where, as above, $A_{\tau_\nu \sigma^{a-b + 1} \tau_{\mu}^{-1} }$ is the zero block if $\tau_\nu \sigma^{a-b + 1} \tau_{\mu}^{-1}$ is not in $H$.

\color{blue}
\begin{remark}
Here the row blocks are indexed by $(\nu, a)$ and the column blocks are indexed by $(\mu, b)$.
\end{remark}
\color{black}

Note that $\tau_\nu \sigma^{a-b + 1} \tau_{\mu}^{-1} \in H$ if and only if $\tau_\nu \sigma^{a-b + 1} \in H \tau_\mu$.
But since $\tau_\nu \sigma^{a-b + 1}$ is in the coset $H \sigma_\nu^{a-b+1}\tau_\nu$ we conclude that the block is zero unless
both~$\mu = \nu$ and~$a-b+1 \equiv 0 \pmod {l_\nu}$.

For a fixed $\nu$ we can consider the square matrix  $C_\nu$
which is described as a block matrix whose $(a, b)$ block is the $\delta$ by $\delta$ square matrix
 $\left( A_{\tau_\nu \sigma^{a-b + 1} \tau_{\nu}^{-1} }\right)$.
In particular the $(a, b)$ block is zero unless $a-b+1 \equiv 0 \pmod {l_\nu}$. Note that $C_\nu$ is an $\l_\nu \delta$ by $\l_\nu \delta$ square matrix.
Then (for a suitable ordering of a basis) one can write $B_\sigma$ in terms of blocks as follows:
$$
B_\sigma = \begin{pmatrix}
C_1 & 0 & \cdots & 0 \\
0 &C_2 & \cdots & 0 \\
\vdots & \vdots &  & \vdots \\
0 & 0 & \cdots & C_r
\end{pmatrix}.
$$
If $a = 0, 1, \ldots, l_\nu -2$ then the $(a, b)$ block of $C_\nu$ is zero unless $b=a+1$, and when~$b=a+1$ the
block is $A_{\tau_\nu \sigma^{a-b + 1} \tau_{\nu}^{-1} }$ which is the $\delta$ by $\delta$ identity matrix $E$.
If~$a = l_\nu - 1$ then the $(a, b)$ block of $C_\nu$ is zero unless $b = 0$ and the $(l_\nu-1, 0)$ block is
given by~$A_{\tau_\nu \sigma^{l_\nu} \tau_{\nu}^{-1} } = A_{ \sigma_\nu^{l_\nu} }  = A_{\sigma_\nu}^{l_\nu}$.
So  $C_\nu$ decomposes into blocks as follows:
$$
C_\nu = \begin{pmatrix}
0 & E & 0& \cdots & 0 \\
0 &0 & E & \cdots & 0 \\
\vdots & \vdots & \vdots &  & \vdots \\
0 & 0 & 0 & \cdots & E\\
A_{\sigma_\nu^{l_\nu}} & 0 & 0 & \cdots & 0
\end{pmatrix}.
$$
The characteristic polynomial in $t$ is then
$$
|E -t B_\sigma| = \prod_{\nu = 1}^r    \left| E - t C_\nu \right| 
= \prod_{\nu = 1}^r 
\begin{vmatrix}
E & - t E & 0& \cdots & 0 \\
0 &E &-t E & \cdots & 0 \\
\vdots & \vdots & \vdots &  & \vdots \\
0 & 0 & 0 & \cdots & - t E\\
- t A_{\sigma_\nu^{l_\nu}} & 0 & 0 & \cdots &E
\end{vmatrix}.
$$
Adding $t$ times the first column to the second, then $t$ times the (new) second to the third, and so on, one gets
$$
|E -t B_\sigma| = \prod_{\nu = 1}^r 
\begin{vmatrix}
E & 0 & 0& \cdots & 0 \\
0 &E &0 & \cdots & 0 \\
\vdots & \vdots & \vdots &  & \vdots \\
0 & 0 & 0 & \cdots & 0\\
- t A_{\sigma_\nu^{l_\nu}} & - t^2 A_{\sigma_\nu^{l_\nu}} & - t^3 A_{\sigma_\nu^{l_\nu}} & \cdots &E - t^{l_\nu} A_{\sigma_\nu^{l_\nu}}
\end{vmatrix}.
$$
Thus
$$
|E - t B_\sigma| 
= \prod_{\nu = 1}^r 
\left| E - t^{l_\nu} A_{\sigma_\nu^{l_\nu}}  \right|.
$$

Note this last formula does not depend on the choice of  $\sigma$ since different choices of Frobenius elements gives the same characteristic polynomials.

The contribution of $\mf p$ to $L(s, \chi_\psi; k)$ is
$$
\frac{1}{|E - (N\mf p)^{-s} B_\sigma|} 
=\prod_{\nu = 1}^r 
\frac{1}{| E - (N\mf p)^{-l_\nu s} A_{\sigma_\nu^{l_\nu}}  |}
=\prod_{\nu = 1}^r 
\frac{1}{| E - (N\mf q_\nu)^{-s} A_{\sigma_\nu^{l_\nu}}  |}.
$$
We have already seen that the Frobenius element associated with $\fP_i$ over $\Omega$ is equal to $\sigma_i^{l_i}$.
So the right hand side of the above formula gives the product of the~$\mf q_\nu$ contributions to $L(s, \psi; \Omega)$.
Thus Satz 1 is proved.

\end{proof}

\chapter{Factorization of Zeta Functions}

Satz 1 gives us, for starters, a factorization of zeta functions of intermediate fields $\Omega$  in terms of primitive $L$-series associated to $K/k$

When we consider the trivial representation and the trivial character $\chi = 1$ (der Hauptcharakter $\chi_1$) we get
$$
L(s, \chi_1; k) = \prod_{\frak p} \frac{1}{| E - \left(N \frak p\right)^{-s} A_{\frak p} |} = \prod_{\frak p} \frac{1}{1 - \left(N \frak p\right)^{-s} } 
$$
which is, up to a finite number of factors, just the zeta function $\zeta_k(s)$ of the base field.

More generally if $\Omega$ is an intermediate field between $k$ and $K$, and if $H$ is the Galois group of $K/\Omega$ with trivial character (Hauptcharakter) $\psi_1$
then
$$L(s, \psi_1; \Omega) = \zeta_\Omega (s),$$
 at least up to a finite number of factors. Let $\Pi_\Omega$ be the induced representation
associated with the trivial representation of $H$. Note that $\Pi_\Omega$ is simply the representation associated with the permutation of cosets of $H$ in $G$ (so if $\Omega$
is itself Galois over~$k$, it corresponds to the regular representation of the Galois group of $\Omega$ over $k$).
Thus the associated character $\chi_\Omega$ has the property that, for any $\sigma \in G$, the value~$\chi_\Omega(\sigma)$ is the number of cosets fixed by $\sigma$ under this action,
so is determined in a most simple manner.  If we decompose~$\chi_\Omega$ in terms of primitive characters
$$
\chi_\Omega (\sigma) = \sum_{i = 1}^x g_i \chi^i(\sigma)
$$
then $g_i$ is obtained using (\ref{E2}):
\begin{equation}\label{E18}
g_i = \frac{1}{n} \sum_{\sigma} \chi_\Omega (\sigma) \chi^i(\sigma^{-1}) 
\end{equation}
($n$ is the order of $G$ and so is $n = [K:k]$).
So Satz 1 in combination with (\ref{E14}) implies
\begin{equation}\label{E19}
\zeta_{\Omega} (s) = \prod_{i=1}^x \left( L(s, \chi^i) \right)^{g_i}
\end{equation}
which is the desired factorization (up to a finite number of factors).

{\color{blue}
\begin{remark}
From what Artin has said up to this point it is  apparent that he regards~$G$ as acting on the left for its natural action on $K$, but regards~$G$ as 
acting on the right for linear representation. Under this convention
$\Pi_\Omega$ is the permutation representation of the right action of $G$ on the collection $H\backslash G$ of right cosets. However, the associated character
$\chi_\Omega(\sigma)$ is the same whether we use left actions or right actions here (in other words,  the number of left cosets fixed by $\sigma \in G$ is
the same as the number of right cosets fixed by $\sigma$).
\end{remark}
}

In the special case of $K = \Omega$ the induced representation is the regular representation and we get the simple formula
\begin{equation}\label{E20}
\zeta_{K} (s) = \prod_{i=1}^x \left( L(s, \chi^i) \right)^{f_i}.
\end{equation}

{\color{blue}
\begin{remark}
Here $f_i$ is the degree of character $\chi_i$. So if $K/k$ is Abelian, we have $f_i = 1$.
In general, all the primitive $L$-series for $K/k$ occur in the factorization.
\end{remark}
}

Formula (\ref{E19}) gives all the relations between the zeta functions of intermediate fields. To get such a relation, one uses (19) for various $\Omega$ and eliminates the $L$-series factors $L(s, \chi^i)$.
One is left with relations between zeta functions. 
So the equations~(\ref{E19}) can be regarded as parameterizing relations. We will show later (Section~8) that this is the only way to get relations between zeta functions (when we reduce to the case~$k=\bQ$).\footnote{See E.~Artin, \"Uber
die Zetafunktionen gewisser algebraisher Zahlk\"orper 
(Concerning the zeta functions of certain algebraic number fields), 
Math. Ann Bd. 89, where the relations in special cases are obtained.}
This essentially solves the problem of relations between zeta functions.

There is another way to formulate our results. Observe that the factorization~(\ref{E19}) of $\zeta_\Omega(s)$ 
runs parallel to the factorization into irreducible polynomials of the  group determinant (Gruppendeterminante) associated to the permutation representation $\Pi_\Omega$.
So one can say the following: 

One gets all the relations between zeta functions of intermediate fields by finding the relations between 
the group determinants associated with transitive permutation actions of $G$, and replacing
the group determinants with the corresponding zeta functions.

{\color{blue}
\begin{remark}
The ``group determinants'' that Artin mentions above are certain homogeneous polynomials associated to groups and their representations.
They are not as familiar today as they were when Artin wrote this paper, so I will give some details.
 They are called ``determinants'' since
they arise as  determinants of matrices with entries
that are homogeneous linear polynomials. These
polynomials were studied by Dedekind and Frobenius, and their study led Frobenius to his theory of characters of non-Abelian groups in 1896
that is in fact the basis of the current paper (see~\cite{Hawkins1971history}).
They are easy enough to define:
consider the polynomial ring~$\bC [X_{g_1}, \ldots, X_{g_n}]$ associated to a given finite group $G = \{g_1, \ldots, g_n\}$ where the $X_{g_i}$
are independent variables. If $g \mapsto A_{g}$ is a representation of $G$ by complex matrices, then the determinant associated to the representation is
simply the determinant of the following matrix:
$$
A_G \defeq \sum_{g \in G} X_g A_g.
$$
The matrix $A_G$ has a particularly nice description if the representation is a permutation representation, and even more so for the regular
representation (it is a good exercise to work these out). The determinant associated to the regular representation is called the ``group determinant'' of $G$ and can be thought of as
a fundamental algebraic invariant of $G$.

The determinant associated to a representation is an irreducible polynomial if and only if the representation is an irreducible representation,
and the decomposition of a representation is reflected in the factorization of its associated determinant. Observe also that the degree of such a determinant polynomial is equal to
the degree of the representation. Note that, historically speaking, the problem of factoring the group determinant proceeds, and in fact motives, the problem of decomposing a representation into irreducible factors that is the starting point of modern representation theory (See \cite{Hawkins1971history}).

For a simple example, the group determinant of a two-element group~$G=\{1, \sigma\}$ is just~$X_1^2 - X_\sigma^2$ which factors
as $(X_1 + X_\sigma) (X_1 - X_\sigma)$, reflecting the fact that the regular representation of $G$ decomposes into two
irreducible representations, each of degree 1.
\end{remark}
}

For now these relations are only valid up to a finite number of factors. Because of the existence of functional equations for zeta functions, we can use the well-known methods of
Herrn Hecke to show the relations are exactly valid.\footnote{E.~Hecke: \"Uber eine neue Anwendung der Zetafunktion auf die Arithmetik
der Zahlk\"orper 
(concerning a new application of zeta functions to the arithmetic of number fields)\color{black}. G\"ottinger Nachrichten 1917.}

{\color{blue}
\begin{remark}
These methods of Hecke allow us to use functional equations of zeta functions to conclude that if a relation between zeta functions is valid up to a finite number
of Euler factors, then the relation holds exactly.
(See Lemma~\ref{fixingL_lemma} below for an illustration of this phenomenon.)
\end{remark}
}

Of course, similar considerations  apply for relations between $L$-Series of intermediate fields.

{\color{blue}
\begin{remark}
We can use Artin's results to get an even more dramatic conclusion.
Suppose $\zeta_\Omega$ is a zeta function with base number field $\Omega$, or more generally consider $L$ functions with base field $\Omega$.
Then by result alluded to at the and of Section~2, we can take $K$ to be an extension of $\Omega$ that is Galois over $\bQ$.  So the above
considerations allow us to express $\zeta_\Omega$ (or more general $L$-functions) in terms of primitive~$L$-functions over $\bQ$.
Artin, in Section 8 below, will  show that this decomposition is unique.
\end{remark}
}

\chapter{The Abelian Case}

We now consider the case where $G$ is Abelian. We investigate whether the primitive~$L$-series defined in this document correspond to the 
usual $L$-series.

{\color{blue}
\begin{remark}
These earlier $L$-series were defined by Weber and generalize those defined by Dirichlet.
They are defined in terms of characters of class groups (where characters are understood here in the traditional Dirichlet-Dedekind sense
as a homomorphism from a finite Abelian group into $\bC^\times$).
\end{remark}
}

When $G$ is Abelian, each conjugacy class has a single element. So for each prime~$\mf p$  of~$k$  not dividing the relative discriminant of~$K/k$ there is exactly one Frobenius
element $\sigma \in G$, and (\ref{E8}) holds for all primes $\mf P$ in $K$ above $\mf p$.
One can replace (\ref{E8}) with the congruence
\begin{equation}\label{E21}
\sigma A \equiv A^{N \frak p} \pmod {\mf p}.
\end{equation}

Further, the irreducible representations of $G$ are all of degree 1, and they 
correspond to the ordinary Abelian characters $\chi^i(\sigma)$ of $G$. Hence
\begin{equation}\label{E22}
L(s, \chi^i) = \prod_{\mf p} \frac{1}{1-\frac{\chi^i (\sigma)}{N \mf p^s}}
\end{equation}
where $\sigma$ denotes the Frobenius element associated to $\mf p$.

{\color{blue}
\begin{remark}
Equation (\ref{E21}) follows from the Chinese remainder theory.  In (\ref{E22})  the Frobenius element $\sigma$ depends on $\mf p$. Artin makes this implicit, but a notation such as $\sigma(\mf p)$ could be used here if we wanted to make this explicit.
\end{remark}
}

Now in this situation $K$ is the class field of a certain 
class group $\{ C_1, \ldots, C_n \}$ for a certain modulus $\mf m$ (a certain ideal of $\mc O_k$)
with the property that a prime ideal~$\mf p$ of $\mc O_k$ prime to $\mf m$ splits into prime ideals of the first degree in $\mc O_K$ if and only if $\mf p$ is in~$C_1$ where~$C_1$ is the identity class (Hauptklasse).\footnote{See 
Teiji Takagi: \emph{\"Uber eine Theorie des relativ Abelschen Zahlk\"orpers}
{(concerning a theory of relative Abelian number fields)}, Journal of the College of Science, Tokyo 1920  \cite{Takagi1920}.
Further reference to Takagi will generally be from this paper.}  

{\color{blue}
\begin{remark}
We can think of  $\{ C_1, \ldots, C_n \}$ as a certain quotient group of the multiplicative group of fractional ideals whose prime factors are prime to $\mf m$.
In other words, each $C_i$ is a class of fractional ideals prime to $\mf m$.
There is a minimal ideal~$\mf m$ that we can use called the conductor, but we get well-defined version of the class group when we use multiplies of this minimal modulus. 
Replacing a modulus by a multiple gives a class group that is naturally isomorphic to the first, so we can 
often say ``the ideal class group'' associated to $K/k$ is we are not concerned about the exact modulus. However replacing a modulus with a multiple can 
reduce the set of prime ideals of $\mc O_k$ prime to the modulus, but only by a finite number.
\end{remark}
}

The identity between our new $L$-series and the usual $L$-series will be shown once we are able to prove the following:

\begin{satz}

\ 

\emph{a)} The Frobenius element $\sigma$ of $\mf p$ depends only on the ideal class $C_i$ containing~$\mf p$, (so we can 
assign a Frobenius element to each ideal class $C_i$ by choosing any prime ideal in that class as a representative).

\emph{b)} This Frobenius map gives an isomorphism between
 the ideal class group and the Galois group~$G$.
\end{satz}

{\color{blue}
\begin{remark}
Observe that if Satz 2 holds for a certain modulus $\mf m$ then it automatically holds for any multiple of $\mf m$.
So there are really two versions of Satz 2, the strong version and the weak version. The strong version  asserts the result  where the class group
is taken with any valid modulus $\mf m$, or equivalently with the conductor as the modulus. The weak version asserts the result for some modulus $\mf m$, or equivalently
asserts (a)  for ``almost all'' prime ideals, i.e. all prime ideals of $\mc O_k$ outside a certain finite subset (and where we can then let $\mf m$ be any valid modulus).

When we know that almost all prime ideals of a given ideal class $C_i$ must have the same Frobenius element, we can conclude that all ideal classes containing
 infinitely many prime ideals can be assigned a well-defined Frobenius element.
But note that every ideal class $C_i$ contains an infinite number of prime ideals $\mf p$ of $\mc O_k$ by a suitable generalization of Dirichlet's theorem concerning
primes in arithmetic progressions. Thus we can assign a Frobenius element to any class. This is the content of the first part of Satz 2. 
\end{remark}
}

This result implies that every character of the Galois group $G$ is then a character of the ideal class group and conversely.
So any $L$-series in our sense is then a $L$-series in the usual sense.
Conversely, if an ordinary $L$-series is given for an ideal class group then it will be an $L$-series for the character of the Galois group 
of the associated class field.
So Satz 2 implies that our new definition is indeed a generalization of the old definition, agreeing with the old definition in the case where $K/k$
is Abelian.

{\color{blue}
\begin{remark}
Satz 2 is called ``Artin reciprocity''. It is the culmination of classical class field theory, and will be proved
by Artin in an article~\cite{Artin1927} appearing a few years later in 1927.
When Artin wrote the current article in 1923, Teiji Takagi had already developed class field theory to a very high degree, and Artin builds on this here. 
Takagi's results give the following.
If $K/k$ is an Abelian extension of degree~$n$ then~$K$ is the class field of a class group~$\mathcal C = \{ C_1, \ldots, C_n \}$ defined with respect to a modulus $\mf m$
for some ideal~$\mf m$ of $\cO_k$. What this means is that $\{ C_1, \ldots, C_n \}$ partitions the collection of ideals, and even fractional ideals, 
of~$\cO_k$ prime to~$\mf m$.
Furthermore, the set $\mathcal C = \{ C_1, \ldots, C_n \}$ of these classes is a group where $C_i C_j$ is defined as the class containing $I_i I_j$ 
for any choice $I_i \in C_i$ and $I_j \in C_j$.
The modulus $\mf m$ is such that all prime ideals $\mf p \in C_i$ prime to $\mf m$ are unramified in $\cO_K$ in the sense that~$\mf p \cO_K$
factors into distinct prime ideals. Furthermore, for such $\mf p$ prime to $\mf m$, 
we have that~$\mf p$ is in the identity class $C_1$ if and only if $\mf p$ splits in $\cO_K$ (in the sense that~$\mf p \cO_K$ factors into $n$ distinct primes of relative degree 1). 

Another very important result of Takagi  is that $\mathcal C = \{ C_1, \ldots, C_n \}$ is isomorphic to the Galois group $G$ of $K/k$.
Interestingly, Takagi showed the isomorphism abstractly and did not supply a particular isomorphism. What Artin reciprocity does is gives a explicit
canonical isomorphism $\mathcal C \to G$.
\end{remark}
}

Satz 2 is also of interest in itself. It gives an explicit description of the isomorphism between the Galois group $G$ and the ideal class group.
In the case where $G$ is cyclic, Satz 2 is completely identical with the general reciprocity law, assuming the base field $k$ has the associated roots
of unity.  And indeed the agreement is so obvious that Satz 2 has to be interpreted as  as the general reciprocity law (even when~$k$ does not have the associated roots of unity) even if the formulation seems a bit strange (fremdartig) at first as a reciprocity law.

{\color{blue}
\begin{remark}
The general reciprocity referred here, and in the next paragraph, seems to be a version developed by Takagi mentioned in special case 5.~below. This law is less familiar today than other reciprocity laws, but
the important take-away is that Takagi's law generalizes the classical reciprocity laws. Since Artin reciprocity generalizes Takagi's reciprocity law
it automatically generalizes all the more familiar classical reciprocity laws.
\end{remark}
}

The situation is, however, 
that our provisional proof of Satz 2 only really succeeds in the cases where the general reciprocity law is accessible to us, that is for~$K$ of prime degree
over $k$ or composite fields of such extensions. For general fields we must, for the time being, just postulate Satz 2. We will do so in
future sections which will allow us to regard all purely Abelian matters as being settled.

In this section we will prove Satz 2 in the cases accessible to us. We will proceed in stepwise fashion where we give the most general results possible in
in order to make the relationships stand out more clearly.

\bigskip
1. \emph{A prime ideal $\mf p$ is in the identity class $C_1$ (the ``Hauptklasse'') if and only if the corresponding Frobenius 
element~$\sigma$ is the identity in $G$.}

{\color{blue}
\begin{remark}
Of course here we are only interested in prime ideals $\mf p$ of $\mc O_k$ prime to the modulus $\mf m$. As we will see in the proof, this result holds  for any valid modulus.
\end{remark}
}

\begin{proof}
If the Frobenius of $\mf p$ is the identity element of $G$ then $A \equiv A^{N \mf p} \pmod {\mf P}$ holds for all $A \in \cO_K$
and all primes~$\mf P$ of $\cO_K$ dividing $\mf p \cO_K$. This implies that the residue field $\cO_K/\mf P$
has $N \mf p$ elements and so the 
degree $[\cO_K/\mf P : \cO_k/\mf p]$ is $1$. Thus~$\mf p \cO_K$ factors into primes of relative degree 1, which means that~$\mf p \in C_1$ by  Takagi  Satz 31.

Conversely, if $\mf p$ is in the identity class $C_1$ then $\mf p \cO_K$ 
factors into  primes of relative degree~1. Thus $A \equiv A^{N \mf p} \pmod {\mf P}$ holds for all prime ideals~$\mf P$ dividing~$\mf p \cO_K$
and all $A \in \cO_K$. 
This means that $\sigma = 1$  works as the Frobenius element.
\end{proof}

{\color{blue}
\begin{remark}
Note that $A \equiv A^{N \mf p} \pmod {\mf P}$ holds for all $A \in \cO_K$ if and only if every element of the residue field $\cO_K/\mf P$ is a root of $x^{Np} - x$.
Lagrange's theorem on the number of roots of a polynomial of a given degree and by Fermat's little theorem, this holds in turn  if and only if $\cO_K/\mf P$ is equal to its subfield $\cO_k/\mf p$.
\end{remark}

\begin{remark}
The above, when combined with Takagi's class field theory, allows us to jump from homomorphisms to isomorphisms.
Suppose in fact that we have a homomorphism $\mathcal C \to G$ from the class group $\mathcal C$ associated to $K/k$ to the Galois group $G$ of $K/k$.
Suppose also that the class of any prime ideal $\mf p$ maps to the associated Frobenius element (perhaps even with a finite number of exceptions). Assume $C \in \mathcal C$ is a class in the kernel. Then 1.~implies that $C$ is the identity class
 (using a density result via
Weber $L$-functions). Thus $\mathcal C \to G$ is injective. From Takagi's class field theory,
$\mathcal C$ and $G$ have the same size (in fact Takagi showed they are isomorphic), thus $\mathcal C \to G$ is surjective as well.
\end{remark}

\begin{remark}
The following result is one where we have to be careful about the distinction between the strong and weak versions of~Satz~2. The proof seems
to give the following: any modulus for which Satz 2 holds for $K/k$ will also yield Satz 2  for~$\Omega / k$ where~$\Omega$ is an intermediate field.
\end{remark}
}

\medskip

2. \emph{If Satz 2 is valid for an Abelian extension $K/k$ then it is valid for $\Omega/k$ for any intermediate field~$\Omega$.}

\begin{proof}
Let $G_\Omega \subseteq G$ be the Galois group of $K/\Omega$. Let $r$ be the order of $G_\Omega$ and let~$s$ be the index
of $G_\Omega$ in $G$. As usual the quotient $G/G_\Omega$ will be identified with the Galois group of $\Omega/K$.

We  assume Satz 2 for the extension $K/k$ so there is a class group $\{C_1, \ldots, C_n\}$ relative to some modulus~$\mf m$,
and a Frobenius isomorphism~$\mc C \to G$ sending $C_i \in \mc C$ to the Frobenius element~$\sigma \in G$ associated to any prime ideal in~$C_i$.

Since $\Omega$ is an intermediate field, by Takagi's results $\Omega$ is the class field for a class group $\mc H = \{ H_1, \ldots, H_s\}$, and moreover $\mc H$ can
be chosen to come from a quotient group of~$\mc C$. In other words, we can use the same modulus $\mf m$ for $\mc H$ as for~$\mc C$, and we can
write the identity class (Hauptklasse) of $\mc H$ as the union of classes of~$\mc C$
 $$
 H_1 = C_1 \cup C_2 \cup \dots \cup C_r
 $$
(where we reindex the elements of $\mc C$ as necessary).\footnote{\label{fn9}\color{blue}
Let $C_1$ be the identity class (die Hauptklasse) of $\mc C$
and let $\mc I_{\mf m}$ be the full group of fractional ideals of $\mc O_k$ relatively prime to the modulus $\mf m$. 
Then 
there is a principle of class field theory similar to what we find in Galois theory: the intermediate fields of $K/k$ are in bijective correspondence
with subgroups of $\mc I_\mf m$ containing $C_1$. This correspondence reverses inclusion.
Given such a subgroup $H_1$ of fractional ideals, the Galois group of the corresponding intermediate extension~$\Omega/k$ is isomorphic to $\mc I_m / H_1$, and the cosets
of $H_1$ in $\mc I_m$ give the class group associated to $\Omega / k$.
Note also that the prime ideals of $H_1$ are exactly the prime ideals in $\mc I_\mf m$ that split in $\Omega$. 
}
By 1.~(above) if $\mf p$ is a prime ideal of $\mc O_k$ not dividing $\mf m$ then the Frobenius element of $\mf p$ in $G/G_\Omega$ (relative to the extension $\Omega/k$)
is equal to the identity coset $G_\Omega \in G / G_\Omega$ if and only if $\mf p \in H_1$.
So by the compatibility of the Frobenius for $K/k$ compared to $\Omega/k$ we have that the Frobenius element of $C_i$ in $G$ is in $G_\Omega$ 
if and only if $C_i \subseteq H_i$.\footnote{\color{blue} See Lemma~\ref{triple_lemma}.}

We now show that all primes in a given class $H_i$ have the same Frobenius element in $G/G_\Omega$. Since $\mc H$ comes from a quotient group of $\mc C$, we can
write $H_i$ as $C'_i H_1$ for some~$C'_i \in \mc C$.  So if $\mf p \in H_i$ is a prime ideal we have $\mf p \in C'_i C_j$ for some $1\le j \le r$ (by the decomposition of $H_1$).
By Satz 1 for $K/k$ we have that $\mf p$ has Frobenius element~$\sigma_i \tau_j \in G$ where $\sigma_i \in G$ is the Frobenius element of the class $C'_i$
and $\tau_j \in G_\Omega$ is the Frobenius element of $C_j$ (recall $C_j \subseteq H_1$ so $\tau_j \in G_\Omega$). Observe that $\sigma_i \tau_j$ is in the coset $\sigma_i G_\Omega$,
and so by the compatibility of the Frobenius elements for $K/k$ compared to $\Omega/k$ we have that the Frobenius element of $\mf p$ in $G/G_\Omega$ is $\sigma_i G_\Omega$.
Thus all primes $\mf p$ in $H_i$ have the same Frobenius element in $G/G_\Omega$. This proves the first part of Satz 1 for $\Omega/k$.

Now we have a well-defined Frobenius function $\mc H \to G/G_\Omega$, and we must show it is a homomorphism.  This follows from the fact
that the following commutes, where the horizontal maps are the Frobenius maps and the vertical maps are the natural quotient maps:
 $$
 \begin{tikzcd} [column sep = normal, row sep = large]
\mc C \arrow[rr] \arrow[d]
& &  G \arrow[d]
\\
\mc H
\arrow[rr]
& & G/G_\Omega
\end{tikzcd}
$$
Since the vertical map $\mc C \to \mc H$ is surjective, and since the top three maps are homomorphisms, the bottom map must also be a homomorphism.

This Frobenius map $\mc H \to G/G_\Omega$ is surjective since if $\sigma G_\Omega$ is in $G/G_\Omega$,
 then~$\sigma$ is the Frobenius in~$G$ for some $\mf p$, which means~$\sigma G_\Omega$ is the corresponding Frobenius in~$G/G_\Omega$.
 Since~$\mc H$ and $G/G_\Omega$ have the same order, the map is in fact an isomorphism.
\end{proof}

{\color{blue}
\begin{remark}
Artin's proof original is a bit terse, so I expanded it a bit in my translation (and even snuck in a commutative diagram not in the original).
I will add extra explanatory details to other proofs as we proceed.
 One
thing Artin did not need to do, however, is to argue that the map $\mc H \to G/G_\Omega$ is surjective since as pointed in a remark after  result 1.~above, we know such a Frobenius map
must be an isomorphism once we know it is a homomorphism. Alternatively, we can see surjectivity right away from the commutative diagram and the fact that the top and right maps are obviously surjective.
\end{remark}

\begin{remark}
In several places in the above proof we used the the compatibility of the Frobenius for $K/k$ compared to $\Omega/k$. 
This was addressed at the end of Section~2 (and is summarized in Lemma~\ref{triple_lemma} in the commentary).
Note that this compatibility is what justifies the commutative diagram that I inserted into the above proof.
\end{remark}

\begin{remark}
The next result can also be regarded as a justification for either the strong or the weak versions of Satz 2. In other words, if the strong version of Satz 2
holds for $K_1$ and $K_2$ then the proof yields the strong version for $K_1 K_2$. If, however, only the weak version of Satz 2 holds for $K_1$ and $K_2$
then the proof can be regarded as a proof for the weak version of $K_1 K_2$. This is based on the observation
that any modulus valid for an Abelian extension is valid for any subextension.
\end{remark}

}

3. \emph{Suppose Satz 2 holds for two Abelian extension  $K_1$ and $K_2$  of $k$ whose intersection is $k$, then it holds for the composite field $K = K_1 K_2$.
.}

\begin{proof} 
Let $\mf m$ be common modulus such that Satz 2 holds for $K_1$ and $K_2$ with modulus~$\mf m$,
and let $C_1, \dots, C_n; D_1, \dots, D_m$ be classes taken for the modulus $\mf m$ where the $C_i$ form the class group for $K_1$
and $D_i$ form the class group for $K_2$.
Let~$G_1$ be the Galois group of $K_1/k$ and $G_2$ be the Galois group of~$K_2/k$.
Suppose that $C_i$ has Frobenius~$\sigma_i \in G_1$ and $D_j$ has Frobenius~$\tau_j \in G_2$.

As we know from Galois theory,  the Galois group of $K_1 K_2 / k$ can be identified with $G_1 \times G_2$.
Note that $K_1 K_2$ is the class field associated to the class group
described by the partition of ideals prime to $\mf m$ given by the intersections $C_r \cap D_s$.\footnote{\color{blue}
This is not too difficult to show. In fact we 
can appeal the the principle of  footnote~\ref{fn9}. Let $\mc I_\mf m$ be the group of fractional ideals of $\mc O_k$ prime to $\mf m$, and let
$\mc P_\mf m$ be the subgroup of principal ideals with totally positive generators congruent to 1 modulo $\mf m$. 
Then $\mc P_\mf m$ corresponds to the ray class
field $L_\mf m$ that clearly contains $K_1 K_2$ since it contains both $K_1$ and $K_2$.
Under the correspondence between subfields of $L_\mf m$ containing $k$ and subgroups of $\mc I_\mf m$ containing $\mc P_\mf m$,
 the group $C_1$ corresponds to $K_1$ and $D_1$ corresponds to $K_2$. 
 So $C_1 \cap D_1$ corresponds to the
smallest subfield of $L_\mf m$ containing both $K_1$ and $K_2$, which is just $K_1 K_2$. The cosets of $C_1 \cap D_1$ in $\mc I_\mf m$ can be seen to be
 the sets $C_r \cap D_s$
as desired. To see this consider the injective homomorphism~$\mc I_\mf m/(C_1 \cap D_1) \to \mc  I_\mf m / C_1 \times\mc I_\mf m/ D_1$ which must be an isomorphism since
$[K_1 K_2 : k] = [K_1 : k ] [K_2 : k]$ (or equivalently, since $C_1 D_1$ corresponds to $k$, the smallest common subfield of $K_1$ and $K_2$, and so must be all of $\mc I_\mf m$).
}
 The product of classes for this class group 
is described by the following equation:
$$
(C_r \cap D_s) (C_u \cap D_v) = C_r C_u \cap D_s D_v.
$$

Now let $A_1 \in \mc O_{K_1}$ and $A_2 \in \mc O_{K_2}$ be generators fo $K_1/k$ and $K_2/k$ respectively. Let $A = \varphi(A_1, A_2)$
be in $\mc O_{K_1 K_2}$. Let $\mf p$ be in $C_r \cap D_s$. Then
$$
A^{N\mf p} \equiv \varphi( A_1^{N\mf p}, A_2^{N\mf p}) \equiv \varphi( \sigma_r A_1, \tau_s A_2) \equiv (\sigma_r, \tau_s) A \pmod \mf p.
$$

Suppose $A$ is an integral  element of $K_1 K_2$ of the form $A_1 A_2$ with $A_1 \in \cO_{K_1}$ and~$A_2 \in \cO_{K_2}$.
If $\mf p$ is a prime ideal in $C_r \cap D_s$ then
$$
A^{N\mf p} = A_1^{N\mf p} A_2^{N\mf p} \equiv (\sigma_r A_1)(\tau_s A_2) = (\sigma_r, \tau_s) A \pmod{\mf p}
$$
Thus the Frobenius of any prime ideal in $C_r \cap D_s$ is $(\sigma_r, \tau_s)$, which is independent of the choice of $\mf p$. So the first part of Satz 2 holds.
The second part follows as well based on what we have shown.
\end{proof}

{\color{blue}
\begin{remark}
In the above proof, Artin does not describe explicity what $\varphi(x, y)$ is, but from context it seems to be a polynomial in $k[x, y]$. Furthermore, 
to support the congruences, the coefficients should
be expressible as fractions of integral elements with denominators not in $\mf p$. Artin does not address the existence of such a polynomial. Fortunately, there
is straightforward way to prove the result that does not rely such a polynomial $\varphi(x, y)$:

As in the above proof, let $\mf p$ be a prime ideal of $C_r \cap D_s$ with Frobenius element~$(\sigma, \tau) \in G_1 \times G_2$.
By Lemma~\ref{triple_lemma}, and thinking of the Galois group of $K_1/k$ as the quotient~$G_1 \times G_2 / G_2$ (with $G_2$ embedded in $G_1 \times G_2$
in the usual way) then the Frobenius element of $\mf p$ for the extension $K_1/k$ is the coset 
$$
(\sigma, \tau) G_2 = (\sigma, 1) G_2.
$$
Under the identification of $G_1 \times G_2 / G_2$ with $G_1$, which identifies the two descriptions of the Galois group of $K_1/k$,
this element is $\sigma$. Thus $\sigma = \sigma_r$ since $\sigma_r$ is the Frobenius element of $\mf p$ for $K_1/k$.
Similarly, $\tau = \tau_s$. Thus the Frobenius of $\mf p$ is~$(\sigma_r, \tau_t)$ as claimed.
\end{remark}
}

Note that because of 3.~(and the structure theorem of finite Abelian groups) we can reduce the proof of Satz 2 to cyclic extensions of 
prime power degree. However, in this paper we will only fully succeed in proving Satz 2 in the case of cyclic extensions of prime degree.

\bigskip

4. \emph{Satz 2 holds for $K = k(\zeta)$ where $\zeta = e^{\frac{2\pi i}{m}}$ is an $m$th root of unity.}\footnote{An analogous proof
can be produced for class fields  of complex multiplication. This shows how the reciprocity laws can be obtained through transcendental 
generators of the class fields.}

\begin{proof} 
Let $\mc C = \{C_1, \ldots, C_n\}$ be a class group associated to the field extension~$K/k$  where, as usual, $C_1$ is the identity class (die Hauptklasse).
For now we allow any  modulus $\mf m$ for $\mc C$, valid for $K/k$, that at least satisfies the following condition:  every prime ideal dividing $m\mc O_k$ also divides $\mf m$ (we will later show that $\mf m = m \mc O_k$ is in fact valid).
The first step is to identify the prime ideals in $C_1$ by determining a splitting law. In other words, we wish to describe which prime ideals $\mf p$ of $\mc O_k$
prime to $\mf m$ have
the property that $\mf p \mc O_K$ factors into distinct primes of relative degree one.

Given such a prime ideal $\mf p$, we know that  $\mf p$ is unramified in $K/k$ and that the distinct $m$th-roots of unity in $\cO_K$ map to distinct 
$m$th roots of unity in the residue field~$\cO_K/\mf P$
for any prime $\mf P$ above~$\mf p$. So $\cO_K/\mf P$ contains all the $m$th roots of unity.
Since $\mf p$ splits in~$\cO_K$, the residue field $\cO_k/\mf p$ 
is isomorphic to~$\cO_K/\mf P$ and so itself contains  all the 
$m$th root of unity. In this case the order~$N \mf p - 1$ of the multiplicative group $(\cO_k/\mf p)^\times$ is divisible by $m$. In other words,~$N \mf p \equiv 1 \pmod m$.

Conversely, suppose $N \mf p\equiv 1 \pmod m$ where $\mf p$ is a prime ideal of $\cO_k$ not dividing~$\mf m$.
Then for each algebraic integer
$A = \alpha_0 + \alpha_1 \zeta + \dots $ in $\cO_K$ (with $\alpha_i \in \cO_k$)
$$
A^{N \mf p} \equiv A \pmod{\mf p}. 
$$
So the residue field~$\cO_K/\mf P$ has size bounded by $N\mf p$, and so equal to $N \mf p$, for all 
primes $\mf P$ above~$\mf p$. Thus $\mf p$ splits in~$\cO_K$.

We have now established our desired splitting law: for prime ideals $\mf p$ of $\cO_k$ prime to $\mf m$, then $\mf p$ splits if and only if $N\mf p \equiv 1 \pmod m$.
So by a fundamental result of class field theory, for prime ideals $\mf p$ of $\cO_k$ prime to $\mf m$, we have $\mf p \in C_1$ if and only if $N\mf p \equiv 1 \pmod m$.

We wish to extend this to showing that $C_1$ consists the of the fractional ideals~$\mf a$ prime to $\mf m$ such that $N \mf a \equiv 1 \pmod m$,
and in fact that all fractional ideals in a given class~$C_i$ have the same norm modulo $m$. It turns out that we can do this by showing that $K$ is contained
in the ray class field of $k$ for modulus $m\mc O_k$, which will allow us to choose $\mf m$ to be~$m\mc O_k$.
So let  $\mc C_m$ be the ray class group of $k$ modulo~$m$.\footnote{\color{blue} 
The ray class group modulo $\mf m$ can be defined as $\mc I_\mf m / \mathcal P_\mf m$ where $\mc I_\mf m$ is the group of fractional ideals prime to $\mf m$
and $\mc P_\mf m$ is the subgroup of principal ideals generated by elements~$\alpha \in k$ such that~$\alpha \equiv 1 \pmod{\mf m}$
and such that $\alpha$ is  positive in all real embeddings of~$k$. It is a basic result that every class of the ray class group contains integral ideals, and in fact prime ideals (by a generalization of Dirichlet's theorem).

The condition
$\alpha \equiv 1 \pmod{\mf m}$ can be interpreted as saying that $\alpha$ is the quotient $\beta/\gamma$ of algebraic integers such
that $\beta$ and $\gamma$ are prime to $\mf m$ and such that $\beta \equiv \gamma \pmod{\mf m}$.
In the current proof we are concerned with the ideal $\mf m = m \mc O_k$, and so we have $\sigma \beta \equiv \sigma \gamma \pmod{m}$
for all $\sigma$ in the Galois group of $K/k$. In particular, 
 $N \beta \equiv N \gamma \pmod {m}$, which we can express as saying that~$N\alpha \equiv 1 \pmod {m}$. 
This is the norm of $\alpha$
as an element of $\bQ$; the norm of the associated principal fractional ideal is the absolute value of the norm of its generator $\alpha$.
Since we assume that~$\alpha$ is positive in all real embeddings of $k$ in~$\bR$, its norm is positive, and so we get that~$N(\alpha\mc O_k) \equiv 1 \pmod m$
where here we mean the norm of the associated principal fractional ideal.}

Suppose $\mf a$ and $\mf b$ are ideals of~$\cO_k$ in the same class in~$\mathcal C_m$.
Then~$\mf a = \alpha \mf b$ for some~$\alpha \in k$ positive in all embeddings of $k$ into~$\bR$ and
such that~$\alpha\equiv  1  \pmod m$. For such $\alpha$ we have
$$
N (\alpha \mc O_k) 
= |N\alpha | = N \alpha  \equiv 1 \pmod m,
$$
so
$$
N \mf a \equiv N \alpha N \mf b \equiv N \mf b \pmod m.
$$
So all the ideals in a given class of $\mc C_m$ have the same norm, and we have a homomorphism~$\mathcal C_m \to (\bZ/m\bZ)^\times$.
Combining classes of norm 1 yields a subgroup $\mc K_m$ of~$\mc C_m$ (the kernel of this norm homomorphism), and 
the quotient  $\mathcal C_m / \mc K_m$ determines a class group with the property that two fractional ideals (prime to $m$) are in the same class if and only if they 
have the same norm. 

In particular, $\mc C$ and $\mc C_m / \mc K_m$  both have the property that (with at most finitely many exceptions) a prime ideal $\mf p$ is in the identity class if
and only if $N \mf p= 1$. According to class field theory this means that the class fields of $\mc C$ and $\mathcal C_m / \mc K_m$ have the same primes that split (with a finite number of possible exceptions), 
and so must be equal.  
So $K$ is the class field of the class group $\mc C_m/\mc K_m$, where this class group is taken to have modulus $m \mc O_k$. We can now fix $\mf m$ to be $m \mc O_k$,
and identify~$\mc C$ with~$\mc C_m/\mc K_m$.
In particular all fractional ideals of a given class $C_i$ have the same norm modulo $m$.

Since $K = k(\zeta)$, we can view the Galois group $G$ of $K/k$ to be a subgroup of~$(\bZ/m\bZ)^\times$ where $\sigma$ is identified with the integer $t$ modulo $m$
for which $\sigma \zeta = \zeta^t$.

Let $C_i$ be a class of $\mc C$, and assume that the fractional ideals of $C_i$ have norm congruent to $n_i$ modulo~$m$. Let $\sigma \in G$ be the Frobenius element
of some prime $\mf p$ of $C_i$.
Since~$\sigma A \equiv A^{N\mf p} \pmod {\mf P}$
for all $A\in \cO_K$ and all primes $\mf P$ in $\mc O_K$ above $\mf p$, we have in
particular that
$$\sigma \zeta \equiv \zeta^{N\mf p} \equiv \zeta^{n_i} \pmod {\mf P}.$$ 
However $\sigma \zeta = \zeta^t$ for some integer~$t$.
So~$\zeta^{n_i} \equiv \zeta^t \pmod {\mf P}$. As mentioned above, distinct $m$-roots of unity in $\cO_K$ map to distinct $m$th roots of unity in the residue field~$\cO_K/\mf P$.
We conclude that $\zeta^{n_i} = \zeta^t$, and so, identifying $G$ with a subgroup of~$(\bZ/m\bZ)^\times$ ,
we see that the Frobenius element is just $n_i \in (\bZ/m\bZ)^\times$.
In particular all primes of $C_i$ share the same Frobenius element, proving the first part of Satz 2. 

By the multiplicativity of the norm map, the Frobenius map is a homomorphism. Observe that the Frobenius map has kernel consisting of the class $C_1$ alone
since only ideals in $C_1$ have norm conguent to 1 modulo $m$. Since $\mathcal C$ and $G$ have the same number of elements (according to Takagi's theory), the induced map $\mathcal C \to G$ is an isomorphism.
\end{proof}

{\color{blue}

\begin{remark}
As mentioned above, it is not really necessary to prove the map is an isomorphism since it being a homomorphism is enough. (See remark after claim~1.).
\end{remark}

\begin{remark}
This gives Satz 2 specifically for modulus $m \mc O_k$.
\end{remark}

\begin{remark}
At this point we know that at least a weak form of Satz 2 holds when~$k= \bQ$ (using 2.,~4.~and the Kronecker-Weber theorem that every finite Abelian extension of~$\bQ$
is a subfield of~$\bQ(\zeta)$ for suitable $\zeta$).
\end{remark}

\begin{remark}
As with other proofs in this translation, the above proof is much expanded and somewhat modified from Artin's original proof in order to make
the argument more accessible to the modern reader. Here we provide more commentary for the proof.
Let $\mf p$ be prime to $\mf m$.
We can use the factorization of the polynomial $x^m - 1$ in $\cO_K [x]$ into linear polynomials and its reduction modulo $\mf p$ to get a factorization into linear
polynomials
$(\cO_K/\mf P) [x]$. Since the deriviative $m x^{m-1}$ is relatively prime to $x^m-1$, the roots in $(\cO_K/\mf P) [x]$ must be distinct
(we know that $m$ is not in $\mf P$ by our assumption on $\mf m$).
This explains  why distinct 
$m$th roots of unity map to distinct roots of unity in the residue field~$\cO_K/\mf P$.

We also used the fact that $(A_1 + A_2)^{N(\mf p)} \equiv A_1^{N(\mf p)} + A_2^{N(\mf p) }\pmod {\mf p}$ for all~$A_1, A_2 \in \mc O_K$. This follows from the fact that $N(\mf p)$
is a power of the characteristic $p$ of $\mc O_k/\mf p$.
\end{remark}
}

\bigskip

{\color{blue} The next  result gives Satz 2 for a class of Kummer extensions:}

\medskip

5. \emph{Suppose $k$ contains the  root of unity $\zeta = e^{\frac{2\pi i}{m}}$ where $m = l^n$ is a power of a prime $l$. Then Satz 2 holds for all
cyclic extensions $K$ of $k$ of degree $m = l^n$.}

\begin{proof} 
It is a standard result of Galois theory\footnote{\color{blue}
See for instance Aluffi~\cite{Aluffi2009}, Chapter VII, Proposition 6.19. In fact, this result is so central to Galois theory that
it was essentially stated by Galois himself in the case that $m$ is prime, but
Galois' argument has a  gap (essentially he fails to show $\mu \ne 0$).
See Edwards~\cite{Edwards1984}, \S 46, Page 63 for a discussion of the gap in Galois' manuscript and a simple way to fix it in a manner that would have been
accessible to Galois himself.}
that any such extension $K$ is of the form~$k\left( \mu^{1/m} \right)$ for some $\mu \in k$ and some fixed choice $\mu^{1/m}$ of $m$th root.
Observe also that the Galois group $G$ can be identified with the group of $m$th roots of unity: the action of $\sigma \in G$ is determined
by  image of $\mu^{1/m}$ which must be of the form $c(\sigma) \mu^{1/m}$ for some $m$th root of unity~$c(\sigma)$. 
The map~$\sigma \mapsto c(\sigma)$
is  our desired isomorphism of~$G$ with the group of $m$th roots of unity.
For convenience we can take $\mu$ to be in~$\mc O_k$ so that $\mu^{1/m} \in \mc O_K$.

Suppose $\mf p$ is a prime ideal of $\cO_k$ prime to $l$. Since $\zeta \in \cO_k$, the
residue field~$\cO_k/\mf p$ has a primitive $m$th root of unity.
In other words $m$ divides~$N \mf p -1$, the order of the multiplicative group of the residue field~$\cO_k/\mf p$.
In other words, $N \mf p \equiv 1 \pmod {m}$.

Therefore,
$$
\left( \mu^{{1}/{m}} \right)^{N\mf p}
\equiv
\mu^{{(N\mf p - 1)}/{m}}  \mu^{{1}/{m}}
\equiv
\left( \frac{\mu}{\mf p}\right) \mu^{{1}/{m}}  \pmod {\mf p}
$$
where $\displaystyle \left( \frac{\mu}{\mf p}\right)$ is the $m$th power character, whose values are $m$th roots of unity.\footnote{\color{blue} 
See for instance Lemmermeyer~\cite{Lemmermeyer2000}, Section 4.1, or Ireland and Rosen~\cite{IrelandR1990}, Section 14.2.}
In particular, the Frobenius element for $\mf p$ is the element of $G$ identified with the $m$th root of unity $\displaystyle \left( \frac{\mu}{\mf p}\right)$.

The essential statement of the general reciprocity law, as given by Takagi, is  exactly that 
$\displaystyle \left( \frac{\mu}{\mf p}\right)$ only depends on the class containing~$\mf p$
(in fact, this holds for any ideal $\mf a$ prime to~$\mu$).\footnote{Takagi~\cite{Takagi1922}.
} 
So let $\mc C$ be a class group for $K/k$ with modulus $\mf m$ (containing~$\mu$ and $l$, say) for which we are certain that 
$\displaystyle \left( \frac{\mu}{\mf a}\right)$ depends only on the class of $\mf a$  in~$\mc C$ for all integral ideals relatively prime to~$\mf m$.\footnote{\color{blue}
Artin's original proof does not specify what modulus $\mf m$ will work here. Perhaps it is $l \mu \mc O_k$ or $m \mu \mc O_k$. In any case, it should be clear by looking at Takagi's paper \cite{Takagi1922}.
Until I have the opportunity to consult Takagi's paper, I will just use any modulus that gets the job done here.

In fact, Artin does not mention the modulus at all in the proof. He also does not specify what~$l$ is, but it is pretty clear from context that $l$ must
at least be a prime. It could be that $l$ is restricted to odd primes. Again, it might require digging into Takagi's paper.
}
It follows now that the first part of Satz 2 holds for such a modulus $\mf m$.
The multiplicativity of the power character implies that the Frobenius map is a homomorphism. From this we get the rest of Satz 2.\footnote{\color{blue}
We get at least the weak form of Satz 2.  Artin also mentions that $\displaystyle \left( \frac{\mu}{\mf p}\right)$ can take on any $m$th root of unity as a value and uses
this to justify surjectivity of the Frobenius map, but as mentioned after result 1. above this follows already from what we have done.
}
\end{proof}

{\color{blue}

\begin{remark}
The above proof is perhaps the most challenging for the modern reader to verify since it relies on results of Takagi that Artin does not 
spell out in detail (nor do the modern sources I have consulted). The power character is well-known though and is easy to define.
Following Section 4.1 of \cite{Lemmermeyer2000},
let $k$ be a number field containing all the $m$th roots of unity where $m$ is a positive integer.
Recall that the reduction mod $\mf p$ map
sends distinct $m$th roots of unity to distinct $m$th roots of unity when $\mf p$ is prime to $m$. By Fermat's little theorem we have
$$
\alpha^{N\mf p - 1} \equiv 1 \pmod {\mf p}
$$
for all $\alpha \in \mc O_k$ outside of $\mf p$, and so
$
\alpha^{(N\mf p - 1)/m}
$
reduces to an $m$th root of unity in the residue field $\mc O_k/ \mf p$.
The power character $\left( \frac{\alpha}{\mf p} \right)_m$ for such an $\alpha$ and $\mf p$ is defined to be the unique $m$ root of unity such that
$$
\left( \frac{\alpha}{\mf p} \right)_m \equiv \alpha^{(N\mf p -1)/m} \pmod {\mf p}.
$$
This can be extended from $\mf p$ to other ideals 
prime to $m$ by defining the symbol to be multiplicative with respect to ideal multiplication.
We can even define  $\left( \frac{\alpha}{\beta} \right)_m$ 
for relatively prime elements nonzero elements $\alpha, \beta \in \mc O_k$, with $\beta$ prime to $m$, by
 considering
the ideal generated by $\beta$.

The power character is central to the study of various reciprocity laws. For example, 
the Eisenstein reciprocity law (\cite{IrelandR1990}, Section 14.2, or \cite{Lemmermeyer2000} Section 11.2) can be elegantly expressed using the power character for the field $k = \bQ(\zeta_l)$:

\begin{thm} [Eisenstein reciprocity]
Suppose $l$ is an odd prime, suppose $\zeta_l$ is a primitive $l$th root of unity, and
suppose $a \in \bZ$ is not divisible by $l$.
If  $\alpha \in \bZ[\zeta_l]$ is relatively prime to $a$, and if $\alpha$ is a primary element (meaning that $\alpha$ is not a unit, is prime to $l$, and is congruent to an element of $\bZ$ modulo~$(1-\zeta_l)^2$)
then
$$
\left( \frac{\alpha}{a} \right)_l  = \left( \frac{a}{\alpha} \right)_l .
$$
 \end{thm}
 
 According to~\cite{Lemmermeyer2000} (Section 11.4), 
 Takagi~\cite{Takagi1922} generalized this result from~$\bQ (\zeta_l)$ to any number field containing~$\zeta_l$.
Apparently Takagi connected such reciprocity laws with his class field theory, which then opened the door to the above result of Artin and to Artin's general reciprocity law.
\end{remark}

\begin{remark}
Because of the close connection between the Frobenius element and the power character illustrated in the above proof, it is common to introduce  reciprocity-like symbols
for the Frobenius. The expression
$$
\left( \frac{K/k} {\mf p} \right)
$$
denotes the Frobenius element associated to $\mf p$. As usual, here $K/k$ is an Abelian extension of number fields, and $\mf p$ is a prime ideal of $\mc O_k$
unramified in $K/k$. This symbol is called the  \emph{Artin symbol} in honor of the ideas introduced in this paper.  When $K/k$ is Galois but not Abelian,
the Frobenius element depends on the choice of prime above $\mf p$, and this leads to the \emph{Frobenius symbol} (introduced by Hasse)
$$
\left[ \frac{K/k} {\mf P} \right]
$$
where $\mf P$ is a prime ideal of $\mc O_K$ unramifield in $K/k$
(See Section 3.2 of \cite{Lemmermeyer2000}).
\end{remark}

}

6. \emph{Suppose $K = k(\alpha)$ is cyclic of degree $r = l^n$ over $k$ where $l$ is a prime, and suppose $\Omega = k(\zeta)$ is an extension of $k$ of degree $m$ 
where $\zeta = e^{{2\pi i}/{l}}$. So $m$ divides~$l-1$ and is necessarily prime to $l$.
If Satz 2 holds for $K^* = \Omega(\alpha)$ over~$\Omega$ then Satz 2  holds for $K$ over $k$.}

{\color{blue}
\begin{remark}
This claim can be generalized, with essentially the same proof, to the following:

\emph{Suppose $K/k$ is an an Abelian extension of degree $r$, and suppose $\Omega/k$ is an Abelian extension of degree
$m$ where $m$ and $r$ are relatively prime. Let $K^* = K \Omega$ be the composite field. If Satz 2
holds for $K^*/\Omega$ then Satz~2 holds for $K/k$ as well.}

In the following translation, Artin's original argument has been adapted to support this more general statement.
\end{remark}
}

\begin{proof} 
We write our Galois groups as $G(K/k), G(K^*/K), G(\Omega/k), G(K^*/\Omega),$ and~$G(K^*/k)$, where $K^*$ is the composite field $K \Omega$.
Since $r$ and $l$ are relatively prime, the intersection of $K$ and $\Omega$ is $k$, and $G(K^*/k)$ can be identified with 
$$
G(K/k) \times G(\Omega/k)
$$
using the usual isomorphisms from Galois theory.
This identification also allows us to identify $G(K/k)$ with $G(K^*/\Omega)$, and $G(\Omega/k)$ with~$G(K^*/K)$. We will write $G$ for both
$G(K/k)$ and $G(K^*/\Omega)$, and we will write $H$ for both $G(\Omega/k)$ with~$G(K^*/K)$. Thus, for example, if $\sigma \in G$
then $\sigma$ can be regarded as an automorphism of $K$ fixing $k$, or as the unique extension of this automorphism to 
an automorphism $K^*$ that fixes $\Omega$.

Fix an ideal $\mf m$ of $\cO_k$ that gives a valid modulus for the class group of~$K^*/k$.
So~$\mf m$ is also a valid modulus for the subextensions $K/k$ and $\Omega/k$. Let~$\cC(K^*/k)$,~$\cC (K/k)$, and~$\cC(\Omega/k)$ be the respective class
groups all using modulus~$\mf m$. By replacing $\mf m$ by a multiple if necessary we can also choose~$\mf m$ so that Satz 2 holds for $K/\Omega$ with modulus $\mf m \mc O_\Omega$

\textbf{Step 1.}
The first step of the proof is to construct a class group $\cC(K^*/\Omega)$ with modulus~$\mf m \cO_\Omega$ together with
an explicit isomorphism  $\cC(K^*/\Omega) \to \cC(K/k)$.
We begin by considering the relative norm map $\mathcal I_\Omega \to \mathcal I_k$ where $\mathcal I_\Omega$ is the group of fractional ideals 
of $\Omega$ prime to $\mf m \cO_\Omega$
and where $\mathcal I_k$ is  the group of fractional ideals of $k$ prime to $\mf m$.
Note that $\cC(K/k)$ can be described as a quotient group of $\mathcal I_k$ and so the composition
$$
\mathcal I_\Omega \to \mathcal I_k \to \cC(K/k)
$$
is a homomorphism. Let $\mf C_0$ be the kernel of this composition.
Observe that \text{if~$\beta \in \Omega$} is prime to $\mf m\mc O_\Omega$ and satisfies $\beta \equiv 1 \pmod {\mf m\mc O_\Omega}$
then the relative norm $N\beta$ in~$k$ must satisfy the congruence~$N \beta \equiv 1 \pmod{\mf m}$ (since $\mf m$ is the intersection
of $\mf m \cO_\Omega$ with $k$).
Furthermore, if  $\beta$ is also positive in all real embeddings of $\Omega$ into~$\bR$ then the relative norm
$N(\beta) \in k$ is totally positive as well.
In particular, the principal ideal generated by such $\beta$ must be in the kernel~$\mf C_0$. This means that the quotient
group $\mathcal I_\Omega/\mf C_0$ yields a class group for modulus~$\mf m \cO_\Omega$.

Let $\mf q$ be a prime ideal of $\cO_\Omega$ prime to $\mf m \cO_\Omega$, and let $\mf p$ be the intersection of $\mf q$
with the subfield $k$. In particular the relative norm $N\mf q \subseteq \cO_k$ is $\mf p^f$ where $f$ divides the relative degree~$m$.
Observe that $\mf q$ splits in $K^*$ if and only if $\mf p$ splits in $K$, since $f$ does not divide $r$.
But $\mf p$ splits in $K$ if and only it is in the identity class of $\cC(K/k)$. Since $f$ is prime to the order of $\cC(K/k)$
this occurs if and only if  $\mf p^f$ in  in the identity class of~$\cC(K/k)$. In other words, $\mf q$ splits in $K^*$
if and only if $\mf q \in \mf C_0$. By Takagi's results this means that $K^*$ is the class field extension of $\Omega$
corresponding to $\mathcal I_\Omega/\mf C_0$. So we write $\cC(K^*/\Omega)$ for $\mathcal I_\Omega/\mf C_0$.
Furthermore, the homomorphism $\mathcal I_\Omega \to \cC(K/k)$ induces an injective homomorphism
$\cC(K^*/\Omega)  \to \cC(K/k)$. Since both groups have order $r$, this map $\cC(K^*/\Omega)  \to \cC(K/k)$,
induced by the relative norm map, is an isomorphism.

\textbf{Step 2.} The second step is to use the isomorphism of step 1, and the assumption of Satz 2 for $K^*/\Omega$,
to define a Frobenius isomorphism on the class group~$\cC(K/k)$. For the isomorphism we will try the composition
$$
\cC(K/k) 
\to
\cC(K^*/\Omega)
\to 
G(K^*/ \Omega)
\to
G(K/k)
$$
where the first map is the inverse of the isomorphism of step 1, the second is the Frobenius isomorphism that exists
by assumption of Satz 2 for $K^*/\Omega$, and the third map is the natural  isomorphism given by restrictions of automorphisms.
This composition is an isomorphism, so to prove Satz 2 for $K/k$ and modulus $\mf m$ we just need to show that this maps the class of a prime ideal $\mf p$
to its corresponding Frobenius element in $G(K/k)$. 

So fix a prime ideal~$\mf p$ of $\cO_k$ prime to~$\mf m$ and 
let~$\mf q$ be a prime ideal in $\cO_\Omega$ above~$\mf p$.
Let~$(\sigma, \tau)$ be the Frobenius element in $G(K^*/k) = G\times H$ corresponding to $\mf p$.
Identifying $G$ and~$H$ with subgroups of $G\times H$, we can write this element as~$\sigma \tau$ and
$$
\sigma \tau A \equiv A^{N\mf p} \pmod {\mf p}
$$
for all $A \in \cO_{K^*}$. 
So $\sigma A \equiv A^{N\mf p} \pmod {\mf p}$ for~$A \in \cO_K$,
and $\tau A \equiv A^{N\mf p} \pmod {\mf p}$ for $A \in \cO_\Omega$
(where here $G$ is identified with $G(K^*/\Omega)$ and $H$ is identified with~$G(K^*/K)$).
Thus $\sigma$ is the Frobenius element of  $\mf p$ in $G=G(K/k)$, and $\tau$ is the Frobenius 
element of $\mf p$ in $H=G(\Omega/k)$.
Note that the relative norm $N \mf q$ is~$\mf p^f$ where $f$ is the order of $\tau$ in $H$. So
$$
\sigma^f A \equiv
\sigma^f \tau^f A  \equiv (\sigma \tau)^f A \equiv A^{{N\mf p}^f}
\equiv  A^{{N\mf q}} \pmod {\mf q}
$$
for all $A \in \cO_{K^*}$ (where here $N \mf q$ is the absolute norm).
Thus $\sigma^f$ is the Frobenius element associated with $\mf q$.
Note that $f$ divides $m = [\Omega: k]$, so $f$ is relatively prime to $r = [K: k]$. This means that
$u f \equiv 1 \pmod r$ for some $u$, and $(\sigma^f)^u = \sigma$.

The isomorphism~$\cC(K^*/\Omega) \to \cC(K/k)$ from step 1
sends the class of $\mf q^u$ to the class of its relative norm~$(\mf p^f)^u$. The class of $\mf p^{f u}$ is 
the class of $\mf p$ since $f u \equiv 1$ modulo~$r$. Thus the inverse isomorphism $\cC(K/k) 
\to
\cC(K^*/\Omega)$
maps the class of $\mf p$ to the class of $\mf q^u$. Since the class of $\mf q$ maps to its Frobenius $\sigma^f$
under the next map~$\cC(K^*/\Omega) \to G(K^*/\Omega)$, the class of $\mf q^u$ maps to 
$$(\sigma^f)^{u} = \sigma^{fu}  = \sigma.$$
Finally, $\sigma$ maps to $\sigma$ under the map $G(K^*/ \Omega)
\to
G(K/k)$ (here we are identifying~$G = G(K/k)$ with $G(K^*/ \Omega)$).

In conclusion the above composition $\cC(K/k) \to G(K/k)$ sends the class of a prime ideal to its Frobenius element.
\end{proof}

\color{blue}
\begin{remark}
This shows that the weak version of Satz 2 for $K^*/\Omega$ implies the weak version of Satz 2 for $K/k$.
\end{remark}
\color{black}

\begin{remark}
We now see that Satz 2 holds for any Abelian extension of degree equal to the product of distinct primes. 
To see this first assume that $K/k$ has prime degree~$l$. Using 5.~we have Satz 2 for $K(\zeta)/k(\zeta)$ where $\zeta$ is a primitive
$l$th root of unity. By 6.~we have Satz 2 for $K/k$ as well. Finally 3.~extends Satz 2 to $K/k$ when~$[K:k]$ the product of distinct primes, or
more generally when the Galois group is the product of cyclic groups of prime order.
\end{remark}

\color{blue}
\begin{remark}
So Artin has proved the following:
\end{remark}
\begin{thm}
Suppose $K/k$ is an Abelian extension of number fields with Galois group~$G$. If $G$ can be factored into cyclic groups of prime order
then the weak form of Satz 2 holds for $K/k$.
\end{thm}
\color{black}

\chapter{Continuation of $L(s, \chi)$ to $\Re(s) \le 1$}

We return to the general case, assuming Satz 2 holds for the Abelian case.
We write~$m(\sigma)$ for the order of an element $\sigma \in G$ of the Galois group of $K/k$.
For each~$\sigma \in G$, let $\mf g^\sigma$ be the subgroup generated by $\sigma$,
and let $\Omega_\sigma$ be the subfield of~$K$ of elements fixed by $\sigma$.
So $\mf g^\sigma$ is the Galois group of $K/\Omega_\sigma$.

Let $\psi_i^{(\sigma)}$ for $i=1, \ldots, m(\sigma)$ be the irreducible characters of the Abelian group~$\mf g^\sigma$
where $\psi_1^{(\sigma)}$ is the trivial character (the ``Hauptcharakter" or the ``principal character'').
If we denote by  $\chi_{\psi^{(\sigma)}_i}$  the induced character of $G$ then, by
Satz 1, equation~(\ref{E15}), we have the following which is valid up to a finite
number of factors in the Euler products:
$$
L\left(s, \psi_i^{(\sigma)} \; ; \; \Omega_\sigma  \right) = L\left(s, \chi_{\psi_i^{(\sigma)}} \; ; \; k  \right).
$$
As in (\ref{E7}), we decompose each $\chi_{\psi^{(\sigma)}_i}$ and obtain the factorizations
\begin{equation} \label{E23}
L\left(s, \psi_i^{(\sigma)} \right) 
=
\prod_{\nu = 1}^x \left( L (s, \chi^\nu) \right)^{r^{(\sigma)}_{i\nu}}
\qquad
(i = 1, 2, \dots, m(\sigma))
\end{equation}
where each $r_{i\nu}^{(\sigma)}$ is a nonnegative integer, and again with validity up to a finite number of factors in the Euler product.
By Satz 2, the left-hand side of (\ref{E23}) corresponds to a traditional $L$-series whose extension to $\bC$ and functional equation was established by Hecke.
We can use the equations (\ref{E23}) to solve for $L(s, \chi^\nu)$ in order to prove the continuation of each~$L(s, \chi^\nu)$.
We can focus on the case $\nu > 1$ since~$L(s, \chi^1) = \zeta_k(s)$ 
is a Dedekind zeta function whose meromorphic continuation is known.\footnote{\color{blue}Actually for any
 $\chi^i$ of degree 1 we have the continuation since that is the case that Hecke considered (assuming Satz 2).}

For $\nu > 1$ we will show that $L(s, \chi^\nu)$ can be expressed in terms of a product of rational powers of the $L\bigl(s, \psi_i^{(\sigma)}\bigr)$ 
where $\sigma$ varies in $G$ and $i$ varies in $\{2, \ldots, m(\sigma)\}$, avoiding the trivial character $\psi_1^{(\sigma)}$.

Because of (\ref{E23}) it suffices to show that the  system of $x$ linear equations
\begin{equation} \label{E24}
\sum_{\sigma \ne 1} 
\sum_{i=2}^{m(\sigma)} 
r^{(\sigma)}_{i\nu} x_i^\sigma 
= \delta_{k \nu} \qquad \nu = 1, 2, \dots, x
\end{equation}
can be solved for each given $k$ in the sequence $2, \ldots, x$.\footnote{\color{blue}
Here Artin is using $k$ as an index. Once we show (\ref{E24}) can be solved, $k$ will return
to its role as denoting the base field.}

{\color{blue}

\begin{remark}
Suppose $x^\sigma_i \in \bQ$ is a solution to the above system of linear equations (for a fixed $k$), then 
\begin{eqnarray*}
\prod_{\sigma\ne 1} \prod_{i=2}^{m(\sigma)}  L\left(s, \psi_i^{(\sigma)} \right)^{x_i^{\sigma}}
&=&
 \prod_{\sigma\ne 1} \prod_{i=2}^{m(\sigma)} \prod_{\nu = 1}^x   L (s, \chi^\nu)^{r^{(\sigma)}_{i\nu} x^\sigma_i}
\\
&=&
\prod_{\nu = 1}^x \prod_{\sigma\ne 1} \prod_{i=2}^{m(\sigma)}  L \left(s, \chi^\nu \right)^{r^{(\sigma)}_{i\nu} x^\sigma_i}\\
&=&
\prod_{\nu = 1}^x  L (s, \chi^\nu)^{\delta_{k\nu}}\\
& =&
L (s, \chi^k).
\end{eqnarray*}
There is a subtlety here: 
the $x_i^\sigma$ are allowed to be rational and so the above quantities are dependent on how the various rational powers
are chosen. Depending on the choices it is possible that the calculation is
 valid only up to a $d$th root of unity where $d$ is a common denominator for the~$x_i^\sigma$.
 So we should think of this equality as holding up to a $d$th root of unity, and as usual up to a finite number of Euler factors.
 
What we can safely say is that $L(s, \chi^k)^d$ can be expressed in terms of a product of integral powers of
the $L\bigl(s, \psi_i^{(\sigma)} \bigr)$ (ignoring a finite number of Euler factors), and so $L(s, \chi^k)^d$ has a meromorphic
continuation to $\bC$. Another way to say this is that there is a meromorphic continuation of $L(s, \chi^k)$
on a Riemann surfaces $\mc L$ mapping onto~$\bC$ with fibers of size bounded by $d$. Or we can take the old point of view that $L(s, \chi^k)$
is a ``multivalued function'' that has an analytic continuation outside a discrete set of singularities.

As we will see, Artin suspects this continuation is single valued (that is, $\mc L$ can be taken to be $\bC$). In other words, 
Artin hoped that $L(s, \chi^k)$ itself, and not a power, has a meromorphic continuation.
This was first proved by R.~Brauer~\cite{Brauer1947a} in~1947. Artin's deeper conjecture that this continuation is actually analytic when~$\chi^k \ne 1$ is
still open.
\end{remark}
}

Now $r_{i1} = 0$ for each $i>1$, so equation~(\ref{E24}) with $\nu =1$ automatically holds.\footnote{
\color{blue}
This follows from (\ref{E6}).}
Thus we only need
to consider $\nu \ge 2$. So in order for (\ref{E24}) to be solvable, it is sufficient that the matrix
$$
\left(  r_{i\nu}^{(\sigma)} \right) \qquad
\sigma \ne 1,
\qquad
i=2, \dots, m(\sigma); 
\qquad
\nu = 2, \dots, x
$$
has rank $x-1$. Here we regard the columns as being indexed by  $(\sigma, i)$ and rows as being indexed by~$\nu$.
For this matrix to have rank $x-1$ it is necessary and sufficient that the $x-1$ rows of this matrix be linearly independent. 
So we just need to to show that the only solution to the system of linear equations 
\begin{equation} \label{E25}
\sum_{\nu=2}^x r_{i \nu}^{(\sigma)} y_\nu = 0
\end{equation}
is the zero solution with $y_\nu = 0$
(where the system contains an equation for each~$(\sigma, i)$ where~$\sigma \ne 1$
and $2 \le i \le m(\sigma)$.)
So assume $y_2, \ldots, y_x$ is a solution to the system.
Fix $\sigma$ and $\tau \in \mf g^\sigma$ (where $\tau = 1$ is allowed), and for each $i$ from $2$ to~$m(\sigma)$
multiply~(\ref{E25}) by $\psi_i^{(\sigma)} (\tau)$. Now sum the resulting equations as $i$ varies:
$$
\sum_{i=2}^{m(\sigma)} \sum_{\nu=2}^x    r_{i \nu}^{(\sigma)} \psi_i^{(\sigma)} (\tau) \, y_\nu = 0.
$$
Using (\ref{E6}) we can simplify this equation, giving 
the following equation for each choice of $\sigma \in G$ and~$\tau \in \mf g^\sigma$:
$$
\sum_{\nu=2}^x \bigl( \chi^\nu(\tau) - r_{1 \nu}^{(\sigma)}  \bigr)  y_\nu 
= 0
$$
or
$$
\sum_{\nu=2}^x  \chi^\nu(\tau)   y_\nu 
=
\sum_{\nu=2}^x r_{1 \nu}^{(\sigma)}  y_\nu.
$$
The right hand side does not depend on $\tau$, so the left hand side has the same value for all $\tau\in \mf g^\sigma$. In particular,
$$
\sum_{\nu=2}^x  \chi^\nu(\tau)   y_\nu = \sum_{\nu=2}^x  \chi^\nu(1)   y_\nu 
$$
for all $\tau \in \mf g^\sigma$. Note the right hand side of this equation does not depend on $\sigma$, and so the left hand side has the same
value for all $\tau \in G$. Call this value $-y_1$, so
$$
\sum_{\nu=1}^x  \chi^\nu(\tau)   y_\nu = 0
$$
for all $\tau \in G$. Using (\ref{E2}), we see that for each $i \in \{1, \ldots, x\}$
\begin{eqnarray*}
0 = 0 \cdot \sum_{\tau \in G} \chi^{i} (\tau^{-1})  &=&\left(\sum_{\nu=1}^x  \chi^\nu(\tau)   y_\nu\right)  \sum_{\tau \in G} \chi^{i} (\tau^{-1}) \\
&=& 
\sum_{\nu=1}^x \left(  \sum_{\tau \in G}   \chi^\nu(\tau) \chi^{i} (\tau^{-1})  \right) y_\nu \\
&=&
\sum_{\nu=1}^x n \delta_{\nu i} y_\nu = n y_i.
\end{eqnarray*}
So $y_i = 0$ for all $i \in \{1, \ldots, x\}$, establishing the linear independence claim, and so the solvability of (\ref{E24}).

We can now express each $L(s, \chi^\nu)$ in terms of Abelian $L$-series, which gives us a way to extend $L(s, \chi^\nu)$
with properties similar to those of traditional $L$-series. For example, if $\chi^\nu$ is not the identity character ($\nu >1$) the expression 
only involves~$L\bigr(s, \psi^{(\sigma)}_i\bigr)$ with~$i \ne 1$, so $L(s, \chi^j)$ is regular and nonvanishing at $s=1$.

%

Now we change our initial definition of $L$-functions. 
A solution to (\ref{E24}) expresses~$L(s, \chi^\nu)$ in terms of a product of rational powers of traditional 
$L$-series but only up to a finite number of factors. We can modify the definition of $L(s, \chi^\nu)$ so that this expression
is an exact equality, and then use (\ref{E14}) to define $L(s, \chi)$ for general characters.
This modified definition changes $L(s, \chi)$ up to a finite number of factors, so all our results that are valid up to a finite
number of factors will continue to hold with the modified definition. 
The resulting $L(s, \chi)$ will analytically continue as a multivalued function on the whole plane $\bC$ minus possibly a discrete set of branch points,
and going around a branch point will only change the value by a root of unity.
The functional equation of Hecke holds for the Abelian $L$-series,
so will yield a functional equation for our new $L$-series. 
This functional equation can be used to show that the definition of our $L$-series is independent of the solution to~(\ref{E24})
used to build our new $L$-series.

{\color{blue}

\begin{remark}
Let $x_i^{\sigma} \in \bQ$ be the numbers occurring in a solution to (\ref{E24}) (where we change~$k$ to $j$ in what follows), then Artin proposes to 
use the resulting relation, originally valid only up to a finite number of Euler factors, as a new, modified definition:
$$
L (s, \chi^j) \; \defeq \;  \prod_{\sigma\ne 1} \prod_{i=2}^{m(\sigma)}  L\left(s, \psi_i^{(\sigma)} \right)^{x_i^{\sigma}}.
$$
This makes $L(s, \chi^j)$ a multivalued function on $\bC$ minus a discrete set of branch points, that is to say it is a
meromorphic function on a Riemann surface covering~$\bC$.
If $d$ is the common denominator of the $x_i^{\sigma}$ then
$$
L (s, \chi^j)^d = \prod_{\sigma\ne 1} \prod_{i=2}^{m(\sigma)}  L\left(s, \psi_i^{(\sigma)} \right)^{d x_i^{\sigma}}
$$
gives an exact equation between meromorphic functions, where the functions on the right
satisfy nice functional equations established by Hecke. 
From this we see that Artin's definition actually gives $L(s, \chi^j)$ as a meromorphic function on a Riemann surface $\mc L$ which covers
$\bC$ with degree bounded by $d$.

If we want to derive a functional equation for this meromorphic function~$L (s, \chi^j)^d$ we need to observe that
we can use the same solution to (\ref{E24}) for writing $L(s, \overline {\chi}^j)^d$ in terms of Abelian $L$-series:
$$
L \left(s, \overline \chi^j\right)^d = \prod_{\sigma\ne 1} \prod_{i=2}^{m(\sigma)}  L\left(s, \overline \psi_i^{(\sigma)} \, \right)^{d x_i^{\sigma}}
$$
where $\overline \chi^j$ denotes the complex conjugate of $\chi^j$, and ${\overline \psi_i^{(\sigma)}}$ denotes the complex conjugate of ${\psi_i^{(\sigma)}}$.
The validity of this can be 
seen by oberving that~(\ref{E6}) and~(\ref{E7}) are well-behaved under complex conjugation, and the induced character of $\overline \psi_i^{(\sigma)}$ satisfies
$$
\chi_{\overline \psi_i^{(\sigma)}} \; = \; \overline{\chi_ {\psi_i^{(\sigma)}}}.
$$
This gives us a version of~(\ref{E23}) for conjugate characters using the same integers~${r^{(\sigma)}_{i\nu}}$ as the original~(\ref{E23}), and so a solution to (\ref{E24}) will work for both $L (s, \chi^j)^d$ 
and~$L \bigl(s, \overline \chi^j\bigr)^d$.

As we will see, the functional equation for Abelian $L$ series is of a form that is closed under
products, so gives a nice functional equation for  $L(s, \chi^j)^d$.  Artin further observes that the functional equation 
forces $L(s, \chi^j)$,
or better $L(s, \chi^j)^d$, 
 to be independent of the solution to (\ref{E24}). In other words, there can be only one definition for $L(s, \chi^j)^d$
that satisfies such a functional equation and agrees with the earlier definition up to a finite number of Euler factors.
\end{remark}
}

The functional equations for the Abelian $L$-series, and hence our new $L$-series, has the following form:\footnote{E.~Landau, \"Uber Ideale und Primideale in Idealklassen \color{blue} (concerning ideals and prime ideals in ideal classes)\color{black}. Math. Zeitschrift Bd. 2, Seite 104, Satz LXVI.}
$$
{L (1-s, \overline \chi^i)}
=
a_i A^s
 \left( \Gamma(s)\right)^{l_i^{(1)}} 
  \left( \cos \frac{s\pi}{2}\right)^{l_i^{(2)}}
  \left( \sin \frac{s\pi}{2}\right)^{l_i^{(3)}}
{L(s, \chi^i)}.
$$
Here $l_i^{(1)}, l_i^{(2)}, l_i^{(3)}$ are rational, and $A$ is a positive real number. Note that $a_i$ depends on
a choice of branch, and $a_i$ may change by a root of unity
if we change the branch.\footnote{\color{blue}We can take $a_i$ to be a true constant and we can take $l_i^{(j)}$ to be integers if we raise both sides of the equation to an appropriate integral power.}

{\color{blue}

\begin{remark}
In verifying these claims it might be best to work with a power $L(s, \chi^j)^d$ that is meromorphic.
As mentioned above, the transformation from $\chi^j$ and $\overline \chi^j$ is well-behaved and we 
can use the same solution to (\ref{E24}) for both $\chi^j$ and $\overline \chi^j$ to get compatible
decompositions.  So since the above functional equation has a form that is closed under powers and products, we get 
a functional equation for $L(s, \chi^j)^d$, and so for $L(s, \chi^j)$ for a choice of branch.
\end{remark}

\begin{remark}
The form of the functional equation for Abelian $L$-series used here by Artin is a bit different than the form it is usually given today, so it is worth a few comments. (I have not consulted Hecke's original paper, nor the paper of Landau cited by Artin. Instead I consulted
 Tate's thesis. Tate was a student of Artin in the 1940s who showed how to replace Hecke's approach with an approach
using harmonic analysis on the id\`eles.)

Suppose $\chi$ is an Abelian character with conductor~$\mf f$. Then Tate's thesis gives a form of the functional equation
(see \cite{CasselsF1967} pages 342--346) that leads naturally to the version used by Artin.
To describe this, let $S$ be a finite set of places of $k$ including all divisors of the conductor $\mf f$ and all Archimedean places. 
Tate shows that
$$
L(1-s, \chi^{-1}) = \left( \prod_{\mf p \in S} \rho_{\mf p} (s) \prod_{\mf p \not\in S} (N {\mf d}_{\mf p})^{s-1/2} 
\; \chi^{-1}( {\mf d}_{\mf p}) \right) L(s, \chi)
$$
where $\rho_{\mf p}(s)$ denotes certain meromorphic functions related to $\chi$ and the place~$\mf p$,  which are
explicitly 
calculated in Tate's thesis (\cite{CasselsF1967}, Pages 317, 319, 322). Here~${\mf d}_{\mf p}$ denotes the local different ideal.
Recall that the absolute norm of the \text{product~$\mf d = \prod \mf d_{\mf p}$} of these ideals gives the absolute discriminant $|\Delta_k|$ of the field~$k$
(where $\mf p$ includes the non-Archimedean primes in $S$).
The functions~$\rho_{\mf p}(s)$ are as follows:
\begin{itemize}
\item
Suppose $\mf p$ is a real place. 
Consider a  principal ideals generated by $\alpha \in k^\times$ such that
(1) $\alpha \equiv 1 \pmod{\mf f}$, (2) $\alpha < 0$ at $\mf p$, and (3) $\alpha > 0$ for all real places not equal to~$\mf p$.
 (Weak approximations assures such an $\alpha$ exists).
Then if $\chi$ has value $1$ on such a principal ideal $\alpha \cO_k$ then
$$
\rho_{\mf p}(s) = \left( \frac{2}{(2\pi)^s} \right) \cos \left( \frac{\pi s}{2}\right) \Gamma(s),
$$
but if $\chi$ has value $-1$ on $\alpha \cO_k$ then
$$
\rho_{\mf p}(s) = -i \left( \frac{2}{(2\pi)^s} \right) \sin \left( \frac{\pi s}{2}\right) \Gamma(s).
$$
\item
If $\mf p$ is a complex place then
\begin{eqnarray*}
\rho_{\mf p}(s) &=& (2\pi)^{1-2s}
\frac{\Gamma(s)}
{\Gamma(1-s)}\\
&=& 2 (2 \pi)^{-2s}  \sin(\pi s) \Gamma(s)^2 \\
& =  & \left( \frac{2}{(2\pi)^s} \right)^2\sin \left( \frac{\pi s}{2}\right)\cos \left( \frac{\pi s}{2}\right) \Gamma(s)^2 
\end{eqnarray*}
(Note this is just $i$ times the product of the two formulas for real $\rho$).
The first equation is essentially the formula from Tate's thesis. The other equations are justified by
the identity 
$$\Gamma(s) \Gamma(1-s) = \frac{\pi}{\sin \pi s}$$
and the double angle identity for the sine function.

\item
If $\mf p$ is a ramified non-Archimedean place with conductor component $\mf f_{\mf p}$ then
$$
\rho_{\mf p}(s) =  N( \mf d_{\mf p} \mf f_{\mf p})^{s-1/2} W_{\mf p}(\chi)
$$
where $W_{\mf p}(\chi)$ is a certain root of unity called the root number.

\end{itemize}

When we substitute these formulas for $\rho_{\mf p}(s)$ and simplify we obtain the formula
$$
L(1-s, \chi^{-1})  = w 
\left( \frac{2}{(2\pi)^s} \right)^{n}
 (N(\mf f) |\Delta_k|)^{s-1/2}  \sin \left( \frac{\pi s}{2}\right)^{n_1} \cos \left( \frac{\pi s}{2}\right)^{n_2} 
\Gamma(s)^n L(s, \chi)
$$
where $n = [k: \bQ]$, where $n_1$ and $n_2$ are two nonnegative integers with $n_1 + n_2 = n$, where $N(\mf f)$ is the norm of the conductor of $\chi$, 
and where $w$ is a root of unity.
\end{remark}

{
\color{blue}

\begin{remark}
Artin makes the observation that the functional equation fixes the $L$-series exactly, not just up to a finite number of factors. 
In other words, the functional equation picks out a canonical representation of a class of $L$-series up to ``finite Euler factor equivalence''.
Undoubtably this principal was well-known when Artin wrote this article, but it might be helpful to supply the details.
The following Lemma helps make this clear and can be proved just by considering the location of zeros and poles.
Note this lemma can be generalized to a much broader class of admissible functional equations, but we will stick to the concrete form given
in the paper.
\end{remark}

\begin{lemma}\label{fixingL_lemma}
Suppose $L_1(s), \widetilde{L_1}(s)$, $L_2(s), \widetilde{L_2}(s)$ are nonzero meromorphic functions on $\bC$ and assume 
the following two conditions. (i)
$$
\frac{L_i (s)}{ \widetilde{L_i}(1-s)} =  a_i \;  A_i^s \; 
\Gamma(s)^{n_i^{(1)}} 
  \left( \cos \frac{s\pi}{2}\right)^{n_i^{(2)}}
  \left( \sin \frac{s\pi}{2}\right)^{n_i^{(3)}}
$$
where $a_i \in \bC^\times$, where $n_i^{(j)} \in \bZ$, and where $A_i$ is a positive real constant. And (ii)
$$
L_2 (s) = P(s) L_1 (s), \qquad \widetilde{L_2}(s) = \widetilde{P}(s) \,  \widetilde{L_1}(s)
$$
where
$$
P(s) = \prod_{i=1}^e (1 - \varepsilon_i p_i^{-s})^{u_i}, \qquad  \widetilde{P}(s)  = \prod_{i=1}^e (1 - \overline\varepsilon_i p_i^{-s})^{u_i}
$$
where the product is over distinct pairs $(\varepsilon_i, p_i)$  composed of a prime $p_i \in \bZ$ 
and a root of unity~$\varepsilon_i \in \bC^\times$, and where $u_i \in \bZ$.
Then necessarily
$$
L_1 (s) = L_2(s)
$$
as meromorphic functions.
\end{lemma}

\begin{proof}
Consider the function
$$
R(s) = \frac{L_2 (s)}{ \widetilde{L_2}(1-s)} \cdot \frac { \widetilde{L_1}(1-s)} {L_1 (s)}  = \frac{P(s)}{ \widetilde{P}(1-s)} 
$$
which, according to the functional equations of assumption (i), has the form
$$
R(s) =  a \;  A^s \; 
\Gamma(s)^{n^{(1)}} 
  \left( \cos \frac{s\pi}{2}\right)^{n^{(2)}}
  \left( \sin \frac{s\pi}{2}\right)^{n^{(3)}}
$$
for some  $a \in \bC^\times$,  $n^{(j)} \in \bZ$, and  $A$ a positive real constant. 
Note that all the zeros and poles of $P(s)$ occur on the line $\Re(s) = 0$ (because $p_i^s$ can equal $\varepsilon_i$ only
on this line),
and all the zeros and poles of $\widetilde P(1-s)$ occur on the line $\Re(s) = 1$.
So the only possible real zeros and poles of $R(s)$ occur when~$s=0$ or~$s=1$; in particular the number of real zeros and poles of $R(s)$ is finite. Since~$\Gamma(s)$ has no zeros and poles for  real~$s > 0$,
this forces~$n^{(2)} = n^{(3)} = 0$ in order to avoid an infinite number of real zeros or poles for~$R(s)$.
Since $\Gamma(s)$ has an infinite number of real poles (at nonnegative integers), we can conclude that $n^{(1)} = 0$ as well.
So $R(s) = a A^s$  has no zeros or poles.

As mentioned above, $P(s)$ and $\widetilde P(1-s)$ have disjoint sets of zeros and poles. Since the quotient $R(s)$  has no zeros or poles, this forces both $P(s)$ and $\widetilde P(1-s)$ to have no zeros or poles.
Consider
$$
P(s) = \prod_{j=1}^e (1 - \varepsilon_j p_j^{-s})^{u_j}
$$
and the zero sets of the factors $1 - \varepsilon_j p_j^{-s}$. We see that $s$ is in the zero set of the~$j$th factor
if and only if $p_j^s = \varepsilon_j$. If $\varepsilon_j = \exp(2\pi r_j i)$ with $r_j \in \bQ$, then
$s$ is in the zero set if and only if $s =0 +  t i$ with
$$
t = 2\pi \; \frac{r_j + k}{\log p_j}
$$
for some $k\in \bZ$.
Observe that if $p_j = p_l$ but $\varepsilon_j \ne \varepsilon_l$, then there can be no common root
to $1 - \varepsilon_j p_j^{-s}$ and $1 - \varepsilon_l p_l^{-s}$ simply because $p_j^s = p_l^s$ cannot
be equal to both~$\varepsilon_j$ and~$\varepsilon_l$. So consider the case where  $p_i \ne p_l$.
A common zero of $1 - \varepsilon_j p_j^{-s}$ and~$1 - \varepsilon_l p_l^{-s}$ would yield a real $t$ with
$$
t = 2\pi \; \frac{r_j + k}{\log p_j} =  2\pi \; \frac{r_l + k'}{\log p_l}
$$
where $k, k'\in \bZ$.  If, in addition, $t$ is no zero, then
we would be able to find two nonzero integers $a, b \in \bZ$ where
$$
\frac{\log p_l}{\log p_j} = \frac{a} {b}
$$
and so $\log p_l^b = \log p_j^a$, or more simply $p_l^b = p_j^a$, a contradiction.
So the only possible common zero of $1 - \varepsilon_j p_j^{-s}$ and~$1 - \varepsilon_l p_l^{-s}$
is $s=0$ (and that occurs only if $\varepsilon_j = \varepsilon_l = 1$).

Thus each factor $1 - \varepsilon_j p_j^{-s}$ of $P(s)$ has a zero that is not a zero of any other factor.
Since $P(s)$ has no zeros or poles this implies that each $u_j = 0$.
So $P(s) = 1$ as desired.
\end{proof}

\begin{corollary}
Aside from a possible multiplication by a root of unity,
the definition of $L(\chi^j, s)$ is independent of the solution to
 (\ref{E24}).
\end{corollary}

\begin{proof}
Let $\left\{ x_i^{\sigma}\right\}$ and $\left\{ \tilde x_i^{\sigma} \right\}$ be two solutions to (\ref{E24}), where we use $j$ for $k$ in (\ref{E24}).
Fix a positive integer $d$ such that each~$ x_i^\sigma d$ and $ \tilde x_i^\sigma d$ is in~$\bZ$ and so $L(\chi^j, s)^d$ is meromorphic 
whether we use the $x_i^{\sigma}$ or the $\tilde x_i^{\sigma}$  to define $L(\chi^j, s)$.
Consider $L_1(s), L_2(s)$ be equal to $L(\chi^j, s)^d$ according to the two expressions 
given by $x_i^{\sigma}$ and $\tilde x_i^{\sigma}$ respectively. 
Note that $L_1(s)$ and $L_2(s)$ agree up to a finite number of Euler factor, and the Euler factors 
where they differ are powers of terms of the form
$$
\frac{1}{1 - \varepsilon N(\mf p)^{-s}}
$$
for some prime ideal $\mf p$ in some number field and some root of unity $\varepsilon$.
So $N(\mf p) = p^l$ for some prime $p\in \bZ$. By factoring the polynomial $1 - \varepsilon X^l$ into linear factors,
we get the following:
$$
\frac{1}{1 - \varepsilon N(\mf p)^s} = \prod_{\mu =1}^l \frac{1}{1 - \varepsilon_\mu p^{-s}}
$$
where each $\varepsilon_\mu$ is a root of unity.

With these ideas we can verify that $L_1$ and $L_2$ satisfy the requirements of the above lemma. So $L_1 = L_2$.
This implies that the two definitions of $L(\chi^j, s)$ differ only by a $d$th root of unity factor.
\end{proof}

\begin{remark}
Note that the above lemma also implies that (\ref{E23}) is an exact equation (up to root of unity), a fact that Artin uses in the calculation of $l_i^{(1)}$.
\end{remark}

}

}

To determine $l_i^{(1)}$ explicitly in the functional equation we use (\ref{E23}) and examine the exponent of $\Gamma(s)$
appearing in the functional equation. This is carried out in the following lemma. In the Abelian case
the exponent of $\Gamma$ is $[k: \bQ]$, and the following lemma shows how to generalize this to the non-Abelian case.

\begin{lemma}
The exponent of $\Gamma(s)$ appearing in the functional equation of $L(s, \chi^i)$ is equal to $f_i [k:\bQ]$ where $f_i$ is the degree
of the representation associated to $\chi^i$. In other words $f_i = \chi^i(1)$.
\end{lemma}

\begin{proof}
By equation (\ref{E23}) we see that the exponent of $\Gamma(s)$ appearing in
the functional equation of $L\bigl(s, \psi_i^{(\sigma)} \bigr)$ is equal to 
$$
\sum_{\nu = 1}^x r_{i \nu}^{(\sigma)} l_\nu^{(1)}
$$
but from the functional equation for Abelian $L$-series we know that this exponent should be the degree $[\Omega_\sigma: \bQ]$. Thus,
for each $\sigma \in G$ and $i = 1, \ldots, m(\sigma)$,
$$
\sum_{\nu = 1}^x r_{i \nu}^{(\sigma)} l_\nu^{(1)} = [\Omega_\sigma: \bQ] = [k: \bQ]  \frac{|G|}{m(\sigma)}.
$$
Multiply by $\psi_i^{(\sigma)} (\tau)$ with $\tau \in \mf g^\sigma$, and sum over $i$:
$$
\sum_{i=1}^{m(\sigma)} \sum_{\nu = 1}^x  l_\nu^{(1)} r_{i \nu}^{(\sigma)} \psi_i^{(\sigma)} (\tau) 
= [k: \bQ]  \frac{|G|}{m(\sigma)} \sum_{i=1}^{m(\sigma)} \psi_i^{(\sigma)} (\tau).
$$
Using (\ref{E6}) on the left and (\ref{E3}) on the right, this equation simplifies as 
$$
\sum_{\nu = 1}^x  l_\nu^{(1)}\chi^{\nu} (\tau) 
= [k: \bQ] |G|  \, \varepsilon_{\tau}
$$
where $\varepsilon_\tau$ is $1$ or $0$ depending on whether $\tau = 1$ or $\tau \ne 1$.

The above equation is independent of $\sigma$, and so applies to all $\tau \in G$.
Now multiply by $\chi^i(\tau^{-1})$ and sum over $\tau \in G$:
$$
\sum_{\tau \in G} \sum_{\nu = 1}^x  l_\nu^{(1)} \chi^{\nu} (\tau) \chi^i(\tau^{-1})
\! = \sum_{\tau \in G} [k: \bQ] |G|  \, \varepsilon_{\tau} 
 \chi^i(\tau^{-1}) =  [k: \bQ] |G|  \,  \chi^i(1) =  [k: \bQ] |G| f_i
$$
but the left-hand simplifies by (\ref{E2}) to give $|G| \, l_i^{(1)}$. So $l_i^{(1)}|G|  = f_i [k: \bQ] |G|$.
In other words, $l_i^{(1)}  = f_i [k: \bQ]$.
\end{proof}

We can determine some of the other constants in the functional equation in a similar manner.

{\color{blue}
\begin{remark}
In the above proof Artin regards equation (\ref{E23}) not as an equation valid up to a finite number of Euler factors but as an exact
equation (or at least up to multiplication by a root of unity). This is justified based because both
sides of (\ref{E23})  satisfy the right type of functional equations (see Lemma~\ref{fixingL_lemma}).

Here is another proof of the above lemma that might be of interest; it does not use the strong version of (\ref{E23}). 
We begin with a special case of~(\ref{E7}):
\begin{equation*}
\chi_{\psi^{(\sigma)}_i} (1) = \sum_{\nu =1}^x r_{i \nu}^{(\sigma)} \chi^{\nu} (1) = \sum_{\nu =1}^x r^{(\sigma)}_{i \nu} f_\nu
\end{equation*}
for each $\sigma$ and each $i = 1, \ldots, m(\sigma)$.
But the degree of the induced representation is just the index $[G: \mf g^{\sigma}]$ so
$$
\sum_{\nu =1}^x r^{(\sigma)}_{i \nu} f_\nu = \chi_{\psi^{(\sigma)}_i} (1) = [G: \mf g^{\sigma}] =  \frac{|G|}{m(\sigma)}.
$$
In what follows let $x_i^{\sigma}$ a solution to (\ref{E24}), (where we use $j$ for $k$ in that equation). Then by 
the previous equation and (\ref{E24})
\begin{eqnarray*}
 \sum_{\sigma\ne 1} \sum_{i=2}^{m(\sigma)}  x_i^\sigma \frac{|G|}{m(\sigma)}
&=&
 \sum_{\sigma\ne 1} \sum_{i=2}^{m(\sigma)} \sum_{\nu =1}^x  x_i^\sigma r^{(\sigma)}_{i \nu} f_\nu\\
 &=&
 \sum_{\nu =1}^x \left( \sum_{\sigma\ne 1} \sum_{i=2}^{m(\sigma)}   x_i^\sigma r^{(\sigma)}_{i \nu} \right) f_\nu\\
  &=&  \sum_{\nu =1}^x    \delta_{j\nu} f_\nu\\
    &=& f_j.
\end{eqnarray*}

With this identity we can easily calculate the exponent of the expression
$$ \left( \frac{2}{(2\pi)^s} \right) \Gamma(s)$$
in the functional equation for $L(s, \chi^j)$.\footnote{\color{blue} As usual, if it  makes matters clearer take a power $L(s, \chi_i)^d$
that is meromorphic on $\bC$ instead of dealing with branches. It is clear how to adapt this argument to such a power.}
From Hecke's functional equation for Abelian $L$-series we have that the contribution from each $L\bigl(s, \psi_i^{(\sigma)} \bigr)$ is equal to 
$$
[\Omega_{\sigma}: \bQ] = [k: \bQ] \; \frac{|G|}{m(\sigma)},
$$
and so the total contribution for $L(s, \chi^j)$ is
$$
\sum_{\sigma\ne 1} \sum_{i=2}^{m(\sigma)} {x_i^\sigma } [\Omega_{\sigma}: \bQ]
=[k:\bQ]  \sum_{\sigma\ne 1} \sum_{i=2}^{m(\sigma)}  x_i^\sigma  \frac{|G|}{m(\sigma)} = [k:\bQ] f_j.
$$

We can argue similarly for the part of the function equation of $L(s, \chi^j)$ coming from factors of the type~$(N(\mf f) |\Delta_k|)^{s-1/2}$
from the Abelian $L$-series factors.
For each~$L\big(s, \psi_i^{(\sigma)} \big)$ we can write this factor as $(B_{i}^\sigma |\Delta_\sigma|)^{s-1/2}$
where $|\Delta_{\sigma}|$ is the absolute discriminant of the field $\Omega_\sigma$ and $B_{i}^\sigma$ is a positive integer. But 
$$
|\Delta_\sigma| = N_\sigma  |\Delta_k|^{[\Omega_\sigma: k]}$$
where $\Delta_k$ is the discriminant of $k$ and $N_\sigma$ is some positive integer.\footnote{\color{blue} 
This is a standard result in algebraic number theory. See, for instance, Neukirch~\cite{Neukirch1999}, Corollary 2.10, page 202. }
So we can write
$$
(B_{i}^\sigma |\Delta_\sigma|)^{s-\frac{1}{2}}
= (B_{i}^\sigma N_\sigma)^{s-\frac{1}{2}}  \left(|\Delta_k|^{s-\frac{1}{2}}\right)^{[\Omega_\sigma: k]}.
$$
In the functional equation of $L(s, \chi^j)$ the terms $(B_{i}^\sigma N_\sigma)^{s-1/2}$ combine to give
$$
\left( \prod_{\sigma \ne 1} \prod_{i=2}^{m(\sigma)} \left( B_{i, \sigma} N_\sigma \right)^{x_i^\sigma}\right)^{s-\frac{1}{2}}
$$
which in Artin's notation is $\alpha_j^{s-\frac{1}{2}}$. Observe that $\alpha_i$ is the product of rational powers of positive integers.

The exponent of $|\Delta_k|^{s-1/2}$ in the functional equation of $L(s, \chi^j)$ will be given by
$$
\sum_{\sigma\ne 1} \sum_{i=2}^{m(\sigma)} {x_i^\sigma} [\Omega_{\sigma}: k]
= \sum_{\sigma\ne 1} \sum_{i=2}^{m(\sigma)}  x_i^\sigma \frac{|G|}{m(\sigma)} =  f_j
$$
where we have used the formula for $f_j$ established above. Thus we get the discriminant factor~$\left(|\Delta_k|^{f_j}\right)^{s-\frac{1}{2}}$.
\end{remark}
}
\medskip

All in all, we get the following:

\begin{satz}
The primitive $L$-series $L(s, \chi^i)$ can be analytically continued to the whole plane aside from a possibly discrete
set of branch points. The  orders of the branch points are (unformly) bounded.\footnote{\color{blue} In other words, $L(s, \chi^i)$ can be meromorphically continued on some Riemann surface covering~$\bC$ of finite degree $d$. In fact, $L(s, \chi^i)^d$ can be meromorphically continued on $\bC$ itself for some positive power $d$.} For $i>1$ 
the continuation of (each branch of)~$L(s, \chi^i)$ is holomorphic and nonzero in a neighborhood of $s=1$.
There are zero-free neighborhoods of the line~\text{$\Re (s)  = 1$}, including a region on the plane
defined by $\sigma \ge 1 - c/\log t$ for some constant $c>0$ (here we write a complex  number as $s = \sigma + i t$ with $\sigma, t \in \bR$).
These \text{$L$-functions} satisfy a functional equation of the form:
\begin{equation} \label{E26}
\frac{L (1-s, \overline \chi^i)}{L(s, \chi^i)}
=
\varepsilon_i
 \left( \frac{2}{(2\pi)^s}\right)^{mf_i}
  \left( \alpha_i |\Delta_k|^{f_i}\right)^{s-\frac{1}{2}}
  \left( \cos \frac{s\pi}{2}\right)^{l_i^{(2)}}
  \left( \sin \frac{s\pi}{2}\right)^{l_i^{(3)}}
 \left( \Gamma(s)\right)^{m f_i}.
\end{equation}
where $\Delta_k$ is the discriminant of $k$, $\alpha_i$ is a product of rational powers of (rational) positive integers,  $\varepsilon_i$
are algebraic integers that depends only on the branch under consideration with $|\varepsilon_i| = 1$, $m = [k:\bQ]$, and $f_i = \chi^i(1)$ is the
degree of the representation associated to the character $\chi^i$. 
Furthermore, $l_i^{(2)}$ and $l_i^{(3)}$ are rational numbers.
\end{satz}

With these types of methods (``Auf demselben Wege'') it should also be possible to establish the single-valuedness of our functions,
of which one can easily convince oneself in special cases. At least one can prove that the branching orders are divisible only by primes dividing $|G|$.

Completely new methods will probably be needed to show that our $L$-Series are analytic on all of $\bC$ (aside from the $L$-series associated
to the trivial character (Hauptcharakter)).

{\color{blue}
\begin{remark}
As mentioned above, the methods of this paper show that $L(s, \chi^i)^d$ is entire for some positive integer $d$. In other words, $L(s, \chi^i)$
can be regarded as a $d$-valued function. Artin mentions here that it should be possible to prove that~$d = 1$ is possible (``die Eindeutigkeit unseren Funktionen''), in other words that $L(s, \chi^i)$
is meromorphic on the whole plane. 
His next sentence means that we should at least be able to find a $d$ such that the only primes dividing $d$ are divisors of $|G|$.
The former claim would have to wait until 1947 when it was proved by R. Brauer~\cite{Brauer1947a},
but later claim is, as Artin says, fairly easy to show: see the following remark.

It is still an open problem however on whether the Artin $L$-series is analytic in general. It has been shown in some cases by Langlands and Tunnell, and
these cases were used by Wiles in his proof of Fermat's Last Theorem. 
\end{remark}

\begin{remark}
We now outline an argument for Artin's claim on branching orders for primes $p$ not dividing the order~$|G|$.
Recall that as part of the proof of the solvability of~(\ref{E24}), Artin shows that 
the matrix~$(r_{i \nu}^{(\sigma)})$ has linearly independent rows. We can reduce this matrix modulo $p$,
and by working in a suitable extension $\bF_q$ of $\bF_p$ (containing roots of unity of order $|G|$) we can mimic the proof
given above for $\bQ$ and show that it also works over $\bF_p$  as long as $|G|$ is
not zero modulo $p$. 

Once we know that the matrix $(r_{i \nu}^{(\sigma)})$ has linearly independent rows modulo~$p$, we can find an $x-1$ by $x-1$ submatrix whose mod $p$ reduction is
nonsingular. In other words, we can find a  $x-1$ by $x-1$ submatrix whose determinant is an integer not divisible by $p$. 
We can then
find a solution to (\ref{E23}) in terms of rational numbers whose denominators are not divisible by $p$. This gives a $d_p$ sheeted cover of $\bC$
such that $L(s, \chi^i)$  is meromorphic on the cover. 

In particular, if one goes around a branch point of $L(s, \chi^i)$ then the value will change value by a multiplicative factor that is a $d_p$-root of unity. In other words, the order of the branch is relatively prime to $p$.
This applies to all primes not dividing~$|G|$ as one goes around a branch point. Let $d$ be the GCD
of all the~$d_p$. Going around any branch point changes the value by a $d$-th root of unity, so
$L(s, \psi^i)^d$ descends to a meromorphic function on $\bC$, and at the same time the only primes dividing $d$
are primes dividing $|G|$.
\end{remark}
}

\chapter{ Conjecture of Frobenius (now called the Chebotaryov Density Theorem)}

With the the result just derived  one can easily confirm a conjecture of Frobenius using Formula (\ref{E12}).\footnote{See
\S 5, Formulas (16) and (18)
of the 1896 work of Frobenius cited in footnote~\ref{FrobFoot}. }

\color{blue}

\begin{remark}
This density conjecture of Frobenius that Artin proves here is what we today call the \emph{Chebotaryov (or Chebotarev) density theorem}.\footnote{\color{blue}
Nikolai Chebotaryov (1894--1947) was a mathematician from Ukraine and Russia. The spelling ``Chebotaryov'' is a transliteration of the Ukrainian version of his name, while
``Chebotarev'' is a transliteration of the Russian version. He was born in Ukraine and was educated at Kyiv University. He later became a professor at Kazan University in Russia in 1928
where he spent the remainder of his career.}
Unbeknownst to Artin, Nikolai Chebotaryov had already proved this result about a year
earlier in~1922
without using these new $L$-series. Artin gives a proof here, but it is requires  Artin's reciprocity (Satz~2)  in order
to be assured that Satz~3 holds. Satz 2 was not fully proved until 1927 when Artin proved his reciprocity law. It is interesting to note that Artin's  1927 proof
of his reciprocity law was inspired by the 1925 German versions of Chebotaryov's  proof of this density theorem that Artin read only after he completed the current paper.
\end{remark}

\color{black}

Not only can you derive the conjecture results, but you can also sharpen them without effort.
From formula  (\ref{E12})
for $\log L(s, \chi)$
it follows from known methods that 
\begin{equation}\label{E27}
\sum_{N \mf p \le x}
\chi^i (\mf p)
=
\delta_{1 i} \mathrm{Li}(x) + O\left( x e^{-a\sqrt{\log x}}\right),
\end{equation}
where $\delta_{11} = 1$, but otherwise $\delta_{1i}=0$.

{
\color{blue}
\begin{remark}
By ``known methods'', Artin is presumably referring to  a combination of methods used to prove Dirichlet's theorem
together with those needed for the the prime number theorem generalized to number fields.
As usual, the error term can be greatly improved 
if one assumes the generalized Riemann hypothesis.

For example, a classical form of the  prime number theorem is that
\begin{equation*}
\pi(x)
=
 \mathrm{Li}(x) + O\left( x e^{-a\sqrt{\log x}}\right),
\end{equation*}
for some $a>0$. (See for instance Theorem 6.9, page 179 of \cite{ MV2007}).
Here $\pi(x)$ is the number of primes in $\bZ$ less than $x$
and 
$$
 \mathrm{Li}(x) = \int_{2}^x \frac{1}{\log t} dt.
$$
The proof of the prime number theorem uses a zero-free region for $\zeta(s)$ similar to that described in Satz~3 for Artin $L$-functions.
(See Theorem~6.6, page 172 of \cite{ MV2007} for the classical zero-free region).
\end{remark}
}

For a real number $x$ and a conjugacy class $C$ of $G$, 
let~$\pi(x, C)$ be the number of prime ideals $\mf p$ of $k$ with $N\mf p \le x$
whose Frobenius class is $C$.

We multiply (\ref{E27}) by $\chi^i(\sigma^{-1})$ where $\sigma \in C$, and sum over $i$. From (\ref{E3})
we get
$$
\frac{|G|}{|C_r|}\, \pi(x, C_r) = \sum_{i=1}^x \sum_{N \mf p \le x} \chi^i (\sigma^{-1}) 
\chi^i (\mf p) = \mathrm{Li}(x) + O\left( x e^{-a\sqrt{\log x}}\right).
$$

\begin{satz}
For a real number $x$ and a conjugacy class $C$ of $G$, 
let~$\pi(x, C)$ be the number of prime ideals $\mf p$ of $k$ with $N\mf p \le x$
whose Frobenius class is $C$.  Then
\begin{equation} \label{E28}
\pi(x, C) = \frac{|C|}{|G|} \mathrm{Li}(x) + O\left( x e^{-a\sqrt{\log x}}\right).
\end{equation}
So the density of prime ideals in the class $C$ is equal to the density of $C$ in $G$.
In particular, in each class $C$ there is an infinite number of prime ideals whose Frobenius class is $C$.
\end{satz}

This theorem is a generalization of Dirichlet's theorem concerning primes in an arithmetic progression, 
which (with the help of our general reciprocity law) can be seen to be a special case.\footnote{\color{blue}Although 
Satz 2 is not fully proved
in this paper, it is proved  in  special cases including that of $\bQ(\zeta) / \bQ$. When we work out the class field theory for $\bQ(\zeta)/\bQ$, we find  that Artin's reciporocity gives a correspondence between
the set of primes of $\bQ$ with a fixed Frobenius in the Galois group of~$\bQ(\zeta)/\bQ$ and the set of primes of $\bQ$
in a certain arithmetic progression. So Satz 5 applied to the fields $\bQ(\zeta) / \bQ$ is really just Dirichlet's theorem.
}
Its true meaning has yet to be clarified (``Seine wahre Bedeutung harrt noch der Aufkl\"arung'').

\chapter{Multiplicative Relations Between $L$-Series}

\begin{satz}
If the base field $k$ is $\bQ$ then there are no multiplicative relations between the primitive $L$-Series.
\end{satz}

\begin{proof}
Suppose $x_i$ are integers such that
$$\prod_{i = 1}^x \left( L(s, \chi^i) \right)^{x_i} = 1.$$
Then by (\ref{E12}), with $k=\bQ$,
\begin{equation*}
 \log L(s, \chi^i) = \sum_{p^\nu} \frac{\chi^i(p^\nu)}{\nu p^{\nu s} }
\end{equation*}
where the sum is over all prime powers $p^\nu > 1$.
So when we sum over the $\chi^i$ we get
$$
\sum_{p^\nu} \left(  \sum_{i=1}^x x_i  {\chi^i(p^\nu)}\right) \frac{1}{\nu {p^{\nu s}} } = 0.
$$

{\color{blue}
\begin{remark}
One can tentavely think of the above equality as holding modulo~$2\pi i \bZ$. But in any case the right hand side is a constant
on the connected set~$\Re(s)>1$. Since the left hand side is a Dirichlet series with  constant term $0$ this forces the right hand side to be $0$ as asserted
(by the uniqueness of coefficients of a Dirichlet series).
\end{remark}
}

By the uniqueness of the coefficients of a Dirichlet series we have
$$
 \sum_{i=1}^x x_i  {\chi^i(p)} = 0
$$
for all primes $p$ (and in fact, the prime power $p^\nu$ coefficients vanish as well).
Recall from (\ref{E10}) that~$\chi^i(p)$ denotes to the value of~$\chi^i$ at the Frobenius of $p$.

By Satz 4, each conjugacy class $C$ of $G$ is the Frobenius class for an infinite number of primes $p$. So
$$
 \sum_{i=1}^x x_i  {\chi^i(\tau)} = 0
$$
for all $\tau \in G$. This implies, in the usual way, that $x_i = 0$ for all $i$.
\end{proof}

{\color{blue}\begin{remark}
The last step is just due to the linear independence of characters. This can be shown using (\ref{E2}).
\end{remark}

\begin{remark}
We assumed $x_i$ were integers in the above proof since that is the main case under consideration, but we can let $x_i$ be complex and use the above argument
to show that the functions~$\log L(s, \chi^1), \ldots,  \log L(s, \chi^x), 1$ are linearly independent over $\bC$.
\end{remark}
}

Satz 5 is not valid for general algebraic number fields $k$ since conjugate prime ideals can undermine the result.
In fact, one can easily construct examples (even with~$[k: \bQ]=2$) in which conjugate characters give rise to the same $L$-series.
(In fact, we will see some examples in Section~\ref{S9} where different characters of a given Galois group can give rise to the same $L$-series).

Based on Satz 5 
we see how to find all the relationships between any finite collection of $\zeta$-functions or $L$-series. 
Find a Galois extension
$E$ of $\bQ$ that contains all the field extensions $K/k$ used to define the zeta and $L$-series that you are interested in. 
We can consider all of our given functions as being defined using characters for~$E/\bQ$,
and all of these can be expressed in terms of primitive $L$-series for $E/\bQ$, which are independent by Satz 5.
We can use elimination to find all the relations between our functions because any additional relations are ruled out by Satz 5.
The remark at the end of Section~2 shows we do not necessarily have to transition to a common $E$
to get our decompositions since extending the field
does not change the decomposition. (The common field was mainly used to prove uniqueness). 
So we have reached a conclusion to the problem of determining multiplicative relations.

\color{blue}
\begin{remark}
Here Artin to the independence of the common Galois extension $E$. If you wish to break the dependency on a common 
Galois $E/\bQ$, you would want a way to identify when two primitive $L$-series
for $E_1/\bQ$ and $E_2/\bQ$ are equal. One way is to agree to classify each primitive $L$-series by the minimal Galois extension $E/\bQ$
for which it arises. In other words, consider only irreducible \emph{faithful} representations of Galois groups with base field~$\bQ$.
Note that we have independence for the infinite collection of such primitive $L$-series over $\bQ$ (and the $\bC$-linear independence
for their repective logarithmic functions). When we combine this section with the results of Section 2 we have the following result:
\end{remark}

\begin{corollary}
Every Artin $L$-series factors uniquely as the product of primitive~$L$-series defined over $\bQ$.
\end{corollary}

\begin{remark}
Above Artin mentions using elimination to find relations. This essentially means using commonplace matrix manipulations on integral matrices.
We describe this in more detail.

Suppose $\ell_1, \ldots, \ell_m$ gives a collection of $L$-series, and $L_1, \ldots, L_t$ are all the primitive~$L$-series
defined over $\bQ$ that arises in the decompositions of $\ell_1, \ldots, \ell_m$. Then we can identify each $\ell_i$
with an element of $\bZ^t$ through the exponents of its decomposition. Consider the $\bZ$-module homomorphism
$$
\Phi\colon \bZ^m \to \bZ^t
$$
sending $(c_1, \ldots, c_m)$ to the element of $\bZ^t$ associated to $\ell^{c_1}_1 \cdots \ell_m^{c_m}$.
Then the kernel of $\Phi$ is a free $\bZ$-module of rank bounded by $m$. The elements in this kernel
give us our relations between $\ell_1, \ldots, \ell_m$, and Satz 5 assures us that these are all the relations.
We can concretely calculate a basis for the kernel, i.e. identify all fundamental relations for $\ell_1, \ldots, \ell_m$, by using
row and column operations on the matrix representing~$\Phi$ to identify the kernel. 
The matrix representing $\Phi$ can be written down as soon as we have decomposed each $\ell_i$: its $j$th column is the exponents
occurring in the decomposition of $\ell_i$
(assuming we multiply the matrix on the left). For example, using (\ref{E29}) below, the fundamental relations among $\zeta, \zeta_5, \zeta_6, \zeta_{10}, \zeta_{12}, \zeta_{15}, \zeta_{20}, \zeta_{30}, \zeta_{60}$ discussed there can be calculated from calculating the kernel of the following:
$$
\begin{bmatrix}
1 & 1 & 1 & 1 & 1 & 1 & 1 & 1 & 1 \\
0 & 0 & 0 & 0 & 1 & 0 & 1 & 1 & 3 \\
0 & 0 & 0 & 0 & 1 & 0 & 1 & 1 & 3 \\
0 & 1 & 0 & 1 & 0 & 1 & 2 & 2 & 4 \\
0 & 0 & 1 & 1 & 1 & 2 & 1 & 3 & 5
\end{bmatrix}.
$$
For example, the column vector $(2, 0, -2, 0, 0, 0, -1, 1, 0)$ is in the kernel and
so gives the relation $\zeta^2 \zeta_6^{-2} \zeta_{20}^{-1} \zeta_{30} = 1$.
\end{remark}

\color{black}

\chapter{Applications to Icosahedral Fields} \label{S9}

Finally we apply these results to icosahedral extensions, the simplest extensions that cannot be obtained through a series of Abelian extensions.
Let $K/k$ be a Galois extension of number fields with Galois group $G$ isomorphic to the icosahedral group.
Observe that Satz 2 holds for intermediate Abelian extensions $K'/k'$. To see this 
observe that Abelian groups of the form $H_1/H_2$, where $H_1$ is a subgroup of $G$ and~$H_2$ is a normal subgroup of $H_1$,
have order dividing $60=2^2 \cdot 3 \cdot 5$. 
The $p$-power part of such a group is a cyclic group of order dividing $p$ for $p=3, 5$, and
the~2-power part of such a group is either cyclic of order dividing $2$, or is the 
(Klein) four groups (``Vierergruppe") since $G$ has no elements of order 4.
(Satz 2 has been proved for Abelian Galois groups that are  products of cyclic groups of prime order.)

{\color{blue}
\begin{remark}
The group $G$ of symmetries of the icosahedron is isomorphic to the alternating group $A_5$, which is a simple group of order 60. 
The subgroups of $A_5$ include cyclic subgroups of the following orders: 1, 2, 3, 5. 
There are also Klein four groups, and dihedral subgroups of order~6 and 10. Finally there are subgroups isomorphic to~$A_4$, and of course $A_5$ itself. 
So there are intermediate fields of degree~$1, 5, 6, 10, 12, 15, 20, 30$ and $60$ over $k$.
Note that two subgroups $A_5$ of the same order are actually conjugate, and so are isomorphic.
This implies that two intermediate subfields of $K/k$ of the same degree over $k$ must be isomorphic, and so have equal zeta functions.
Artin uses the notation $\Omega_n$ for a field of degree $n$ over~$k$, and~$\zeta_n$ for its zeta function. We let $\zeta$ be the zeta function of 
the base field~$k$, so~$\zeta = \zeta_1$.
\end{remark}}

In $G$ we have 5 conjugacy classes $C_1, C_2, C_3, C_4, C_5$ with $1, 15, 20, 12, 12$ elements respectively.  The densities of prime ideals in these classes must be 
$$
\frac{1}{60}, \frac{1}{4}, \frac{1}{3}, \frac{1}{5}, \frac{1}{5}.
$$
Furthermore, by the theory of characters developed by Frobenius, we have five simple characters of $G$ and their degrees are $1, 3, 3, 4, 5$. We call the associated
primitive $L$-series~$\zeta$, $L_3^{(1)}, L_3^{(2)}, L_4, L_5$.

We easily get the following factorizations using our methods (where the index refers to the degree of the field over $k$):

\begin{eqnarray} \label{E29}
\zeta_5 & = & \zeta \; L_4 \\
\zeta_6 & = & \zeta\; L_5 \nonumber\\
\zeta_{10 }& = & \zeta\; L_4 L_5 \nonumber\\
\zeta_{12} & = & \zeta \;L_3^{(1)}  L_3^{(2)} L_5   \nonumber\\
\zeta_{15} & = & \zeta\; L_4 \, (L_5)^2 \nonumber\\
\zeta_{20} & = & \zeta \; L_3^{(1)}  L_3^{(2)} \left( L_4\right)^2 L_5\nonumber\\
\zeta_{30} & = & \zeta \; L_3^{(1)}  L_3^{(2)} \left( L_4\right)^2 \left( L_5 \right)^3 \nonumber\\
\zeta_{60} & = & \zeta \; ( L_3^{(1)}  L_3^{(2)} )^3 \left( L_4\right)^4 \left( L_5 \right)^5 \nonumber
\end{eqnarray}

\bigskip

{\color{blue}
\begin{remark}
Verifying (\ref{E29}) is an exercise. One way  to verify it is to do the following:
\begin{itemize}
\item
Identify all subgroups of $A_5$, and the size of the intersections with each of the conjugacy classes $C_1, \ldots, C_5$.
\item
Derive explicit formulas for the five simple characters of $A_5$.
\item
Use  Frobenius reciprocity to calculate the induced characters of trivial characters in terms of the irreducible characters of $A_5$.
\item
Use (\ref{E14}) and Satz 1.
\end{itemize}
\end{remark}}

\begin{proposition} \label{prop_4}
If $G = \mathrm{Gal}(K/k)$ is the icosahedral group, then all the $L$-series associated to representations of $G$ are meromorphic. In other words,
they are single valued (outside of poles) when extended to~$\bC$.
\end{proposition}

\begin{proof}
It is enough to verify this for irreducible representatives. The function $\zeta$ and each $\zeta_n$ are already known to be meromorphic (Hecke).
Note that  the first two equations of~(\ref{E29}) show that $L_4$ and $L_5$ are meromorphic. 

Observe that $K$ is cyclic of degree 5 over an intermediate field $\Omega_{12}$ of degree 12 over $k$. There are four
 primitive $L$-series (in addition to $\zeta_{12}$) associated to $K/\Omega_{12}$ and these can all be expressed in terms of our primitive $L$-series. 
 They are also known to be entire (Hecke). Note that $\zeta_{60}$ factors as $\zeta_{12}$ times the product of these four $L$-series. 
Comparing the expressions for $\zeta_{12}$ and $\zeta_{60}$ 
in (\ref{E29}), we see that the product of these four $L$-series for $K/\Omega_{12}$ is
$$
 (L_3^{(1)}  L_3^{(2)} )^2 \left( L_4\right)^4 \left( L_5 \right)^4.
$$
Each of these $L$-series is based on a one-dimensional representative of $\mathrm{Gal}(K/\Omega_{12})$, and so can be
expressed in terms of a degree $12$ induced representation of~$G$. 
So each decomposition gives 12 as the sum in terms of the integers 3, 4, 5.
We conclude that $12$ is decomposed as $3 + 4 + 5$ for each (note that $5$ cannot occur twice in any one of the sums, so $5$ must occur exactly once
in each sum).
We conclude that two of these $L$-series factor as $L_3^{(1)} L_4 L_5$ and the other two as $L_3^{(2)} L_4 L_5$. (By the way, this gives an example of a field with identical primitive $L$-series where conjugate characters give the same $L$-series)\footnote{\color{blue}Note that any five cycle and its inverse
are in the same conjugacy class of $A_4$, which by Frobenius reciprocity implies that $L$-series for conjugate characters for $K/\Omega_{12}$ will
have the same decomposition.}. So $L_3^{(1)}$ and $L_3^{(2)}$ are meromorphic.
\end{proof}

The above proof also shows the following:

\begin{proposition}
The functions $L_3^{(1)} L_4 L_5$ and $L_3^{(2)} L_4 L_5$ are entire.
\end{proposition}

We can identify other entire function:

\begin{proposition}
The function $L_5$ is entire.
\end{proposition}

\begin{proof}
We have an intermediate extension $\Omega_5/k$ of degree $5$, and an intermediate extension $\Omega_{15}/k$ of degree 15
such that $\Omega_{15}$ is a cyclic cubic extension of $\Omega_5$. This gives us two nontrivial degree 1 characters of $\mathrm{Gal}(K/\Omega_{5})$
of order 3,
whose induced characters are  degree 5 characters of $G$. The associated $L$-series are entire since they are Abelian $L$-series with nontrivial characters.
The product of these series is~$\zeta_{15}/\zeta_5 = (L_5)^2$.
Since the associated induced characters of $G$ are of degree 5, both $L$-series must be equal to $L_5$. So $L_5$ is entire.
(And this gives another example where distinct representations gives the same $L$-series).\end{proof}

\begin{proposition}
The products $L_3^{(1)}  L_3^{(2)}$ and $L_3^{(1)}  L_3^{(2)} L_4$ are entire.
\end{proposition}

\begin{proof}
The proof is similar to the last proof. The first comes from using the quadratic extension~$\Omega_{12}/\Omega_6$ where one produces an entire
$L$-series equal to $\zeta_{12}/\zeta_{6}$.
The second comes from using the quadratic extension~$\Omega_{20}/\Omega_{10}$ where one produces an entire
$L$-series equal to $\zeta_{20}/\zeta_{10}$.
\end{proof}

\begin{remark}
We do not get anything essentially new from the other Abelian extensions.
So $L_5$ and the combinations $L_3^{(1)} L_4 L_5, L_3^{(2)} L_4 L_5, L_3^{(1)} L_3^{(2)}$, and $L_3^{(1)} L_3^{(2)} L_4$
(and their products) are the only $L$-series we can prove are entire.
\end{remark}

{\color{blue}

\begin{remark}
Let's look at the other Abelian subextensions in addition to those treated in the above two Propositions:
\begin{itemize}
\item
$\Omega_{30}/\Omega_{15}$ gives the entire function $\zeta_{30} / \zeta_{15} = L_3^{(1)}  L_3^{(2)} L_4 L_5$.

\item
$K/\Omega_{30}$ gives the entire function $\zeta_{60} / \zeta_{30} = \left( L_3^{(1)}  L_3^{(2)}  L_4\right)^2 \left( L_5 \right)^2$.

\item
$K/\Omega_{20}$ gives the entire function $\zeta_{60} / \zeta_{20} = \left( L_3^{(1)}  L_3^{(2)} L_4\right)^2 \left( L_5 \right)^4$
which must factor into two Abelian (and so entire) $L$-functions
corresponding to the two nontrivial characters $\psi$ and $\psi^{-1}$ of the corresponding cyclic Galois group $H_3$ of order~$3$.
It turns out that the $L$-series associated to $\psi$ and $\psi^{-1}$ are equal and so are both  $L_3^{(1)}  L_3^{(2)}  L_4 \left( L_5 \right)^2$.
To see this, note that the decomposition depends on the the multiplicities of the simple characters $\chi^i$ in the corresponding induced representations, and
these can be calculated using Frobenius reciprocity:  the multiplicities are respectively
$$
\left< \psi, \mathrm{res}\,  \chi^i\right>_{H_3}, \quad\text{and}\quad
\left< \psi^{-1}, \mathrm{res}\, \chi^i\right>_{H_3}
$$
where in these inner products we restrict $\chi^i$ to $H_3$. However, a three cycle in~$A_5$ and its inverse are in the same conjugate class of $A_5$ and so have the same value under $\chi^i$, which means that these two inner products
are actually given by the same sum. This shows the multiplicities are the same.
 
\item
$K/\Omega_{15}$ gives the entire function  $\zeta_{60} / \zeta_{15} = ( L_3^{(1)}  L_3^{(2)} )^3 \left( L_4\right)^3 \left( L_5 \right)^3$
that factors into three entire functions coming from $\Omega_{30}/\Omega_{15}$ extensions.
Looking at the earlier case $\Omega_{30}/\Omega_{15}$, we see these three functions are each~$L_3^{(1)}  L_3^{(2)} L_4 L_5$.
\item $K/\Omega_{12}$ was treated above (Proposition~\ref{prop_4}).
\end{itemize}
\end{remark}
Observe that these entire functions are all just products of the entire products already considered; nothing new.
}

\begin{remark}
Of the zeta functions from (\ref{E29}), we see that the following are divisible by~$\zeta$ with entire quotient: $\zeta_{6}, \zeta_{12}, \zeta_{30}, \zeta_{60}$.
(This leaves the other half in question, namely $\zeta_{5}, \zeta_{10}, \zeta_{15}, \zeta_{20}$).
On can also verify immediately the relations between zeta functions from my earlier article~\cite{Artin1923}.
\end{remark}

{\color{blue}
\begin{remark}
This shows Artin's interest in the following question: if $K/k$ then is $\zeta_K/\zeta_k$ entire?
This helps motivate Artin's conjecture that primitive $L$-series that are not zeta functions are entire.

Artin's earlier article~\cite{Artin1923}, published in 1923, has some interesting relationships between these zeta functions 
in the current case of $G$ isomorphic to $A_5$, the icosahedral group. Some of these include
$$\zeta_{20}\, \zeta^2 = \zeta_5^2 \, \zeta_{12}, \qquad\text{and} \quad 
\zeta_{30}\, \zeta^2 = \zeta_{6}^2 \, \zeta_{20}.
$$
These are immediate given (\ref{E29}) above.
\end{remark}
}

\textbf{Hamburg, Mathematics Seminar, July 1923}

%
%
%
%
%
\bibliographystyle{plain}

\bibliography{../NumberTheoryClassics,../AitkenBooks,../MathHistoryArticles}

\end{document}